\documentclass{amsart}
\usepackage[utf8x]{inputenc}
\usepackage{amsmath,amsthm,amsfonts,amssymb,latexsym, enumerate}
\usepackage[matrix, arrow]{xy}
\xyoption{arrow}
\subjclass{20C20, 20C11}

\begin{document}

\theoremstyle{plain}
\newtheorem{thm}{Theorem}[section]
\newtheorem{lem}[thm]{Lemma}
\newtheorem{pro}[thm]{Proposition}
\newtheorem{cor}[thm]{Corollary}
\newtheorem{statement}[thm]{}

\theoremstyle{definition}
\newtheorem{que}[thm]{Question}
\newtheorem{rem}[thm]{Remark}
\newtheorem{defi}[thm]{Definition}
\newtheorem{Question}[thm]{Question}
\newtheorem{Conjecture}[thm]{Conjecture}
\newtheorem{exa}[thm]{Example}
\newtheorem{Notation}[thm]{Notation}

\def\CE{{\mathcal{E}}} \def\uCE{\underline{{\mathcal{E}}}}
\def\CF{{\mathcal{F}}}  
\def\CG{{\mathcal{G}}}
\def\CH{{\mathcal{H}}}
\def\CL{{\mathcal{L}}} \def\uCL{\underline{{\mathcal{L}}}}
\def\CO{{\mathcal{O}}}
\def\CP{{\mathcal{P}}}
\def\CR{{\mathcal{R}}}
\def\CT{{\mathcal{T}}} \def\uCT{\underline{{\mathcal{T}}}}

\def\OD{{\mathcal{O}D}}
\def\OG{{\mathcal{O}G}}
\def\OGb{{\mathcal{O}Gb}}
\def\OHc{{\mathcal{O}Hc}} \def\tenOHc{\otimes_{\mathcal{O}Hc}}
\def\OH{{\mathcal{O}H}}
\def\OL{{\mathcal{O}L}}
\def\OLd{{\mathcal{O}Ld}}
\def\ON{{\mathcal{O}N}}
\def\OP{{\mathcal{O}P}}
\def\OPE{{\mathcal{O}(P\rtimes E)}} \def\kPE{{k(P\rtimes E)}}
\def\OQ{{\mathcal{O}Q}}
\def\OR{{\mathcal{O}R}}

\def\Br{\mathrm{Br}} 
\def\chr{\mathrm{char}}
\def\codim{\mathrm{codim}}
\def\coker{\mathrm{coker}}
\def\dim{\mathrm{dim}}
\def\End{\mathrm{End}}
\def\foc{\mathfrak{foc}}
\def\Hom{\mathrm{Hom}}
\def\IBr{\mathrm{IBr}} 
\def\Im{\mathrm{Im}} 
\def\Ind{\mathrm{Ind}} 
\def\Irr{\mathrm{Irr}} 
\def\Id{\mathrm{Id}} 
\def\id{\mathrm{id}} 
\def\lcm{\mathrm{lcm}} 
\def\Aut{\mathrm{Aut}} 
\def\Inn{\mathrm{Inn}}
\def\Ker{\mathrm{Ker}} 
\def\mod{\mathrm{mod}} 
\def\modh{{\operatorname{mod-\!}}}
\def\Mod{\mathrm{Mod}} 
\def\Out{\mathrm{Out}}
\def\Pic{\mathrm{Pic}} \def\uPic{\underline{\mathrm{Pic}}}
\def\Piccent{\mathrm{Piccent}}
\def\rank{\mathrm{rank}}
\def\Res{\mathrm{Res}} 
\def\op{\mathrm{op}}
\def\rk{\mathrm{rk}} 
\def\SL{\mathrm{SL}}
\def\soc{\mathrm{soc}}
\def\Syl{\mathrm{Syl}} 
\def\Tr{\mathrm{Tr}} 
\def\tr{\mathrm{tr}} 
\def\Gal{\mathrm{Gal}} 
\def\ten{\otimes}
\def\tenA{\otimes_{A}} \def\tenB{\otimes_{B}}
\def\tenC{\otimes_{C}} \def\tenD{\otimes_{D}}
\def\tenO{\otimes_{\mathcal{O}}}
\def\tenOP{\otimes_{\mathcal{O}P}}
\def\tenOQ{\otimes_{\mathcal{O}Q}}
\def\tenOR{\otimes_{\mathcal{O}R}}
\def\tenK{\otimes_K}
\def\tenk{\otimes_k}
\def\tenkP{\otimes_{kP}}
\def\tenkR{\otimes_{kR}}

\def\C{{\mathbb C}}  
\def\F{{\mathbb F}}               
\def\Q{{\mathbb Q}}               
\def\Z{{\mathbb Z}}               
\def\Fp{{\mathbb F_p}}          
\def\Fq{{\mathbb F_q}}
\def\frakp{{\mathfrak{p}}}

\title[Picard groups of blocks] {On Picard groups of blocks of
finite groups}
\author{Robert Boltje, Radha Kessar, Markus Linckelmann}
\address{Department of Mathematics, City, University  of London 
EC1V 0HB, United Kingdom}
\email{radha.kessar.1@city.ac.uk}

\thanks{This project was initiated while the second and third authors visited the University of California, Santa Cruz, in Spring 2017. They would like to thank the Department of Mathematics  for its  support.  Further work was   supported by the National 
Science Foundation under Grant No. DMS-1440140 while the authors were in 
residence at the Mathematical Sciences Research Institute in Berkeley, 
California, during the  Spring  2018 semester. The third author 
acknowledges support from EPSRC grant EP/M02525X/1.}  

\begin{abstract}
We show that the subgroup of the Picard group of a $p$-block of a finite 
group given by bimodules with endopermutation sources modulo the 
automorphism group of a source algebra is determined locally in terms of 
the fusion system on a defect group. We show that the Picard group of a 
block over the a complete discrete valuation ring $\CO$ of 
characteristic zero with an algebraic closure $k$ of $\Fp$ as residue 
field is a colimit of finite Picard groups of blocks over $p$-adic 
subrings of $\CO$. We apply the results to blocks with an 
abelian defect group and Frobenius inertial quotient, and specialise
this further to blocks with cyclic or Klein four defect groups. 
\end{abstract}

\maketitle

\bigskip

\begin{center}
\today
\end{center}


\section{Introduction}     

Throughout the paper,  $p$  is  a prime number.
Let $\CO$ be a complete local principal ideal domain with residue
field $k$ of characteristic $p$. Assume that either $\CO=k$ or that
$\CO$ has characteristic zero. By a {\it block}  of $\OG$ for $G$ a 
finite group we mean a  primitive idempotent of the  center of the 
group algebra $\OG$. For $B$ an $\CO$-algebra, we denote by $\Pic(B)$
the {\it Picard group of} $B$; that is, $\Pic(B)$ is the group of 
isomorphism classes of $(B,B)$-bimodules inducing a Morita equivalence, 
with group product induced by the tensor product over $B$. 
If $B$ is symmetric (that is, $B$ is free of finite rank as an 
$\CO$-module and $B$ is isomorphic to its $\CO$-dual $B^*$ as a 
$(B,B)$-bimodule) and if $M$ is a $(B,B)$-bimodule inducing a Morita 
equivalence, then the inverse of its isomorphism class $[M]$ in 
$\Pic(B)$ is the class $[M^*]$ of the $\CO$-dual $M^*$ of $M$. 
If $G$ is a finite group, then $\OG$ and the block algebras of $\OG$ 
are symmetric. 

Let $G$, $H$ be finite groups, and let $b$, $c$ be blocks of $\OG$, 
$\OH$, respectively. We say  that a Morita equivalence between $\OGb$ 
and $\OHc$ {\it has endopermutation source} if it is given by  an 
$(\OGb,\OHc)$-bimodule $M$ which  has an endopermutation $\OR$-module 
as a source for some vertex $R$ of $M$, where $M$ is regarded as an 
$\CO(G\times H)$-module.
By \cite[7.4, 7.6]{Puigbook}, a Morita equivalence between $\OGb$ and 
$\OHc$ with endopermutation source  induces an isomorphism 
$\alpha : P\cong$ $Q$ between defect groups $P$ and $Q$ of $b$ and $c$, 
respectively, such that the twisted diagonal subgroup $\Delta\alpha = $ 
$\{(x,\alpha(x))\ |\ x\in P\}$  is a vertex of $M$. Moreover, $\alpha$ 
induces an isomorphism between fusion systems of $b$ and of $c$ on $P$ 
and $Q$, respectively (for some suitable choice of maximal Brauer 
pairs). Set $B=$ $\OGb$. By \cite[Proposition 9]{ZhouYY}, the Morita 
equivalences on $B$ with endopermutation source form a subgroup of 
$\Pic(B)$, which we will denote by $\CE(B)$. We denote by 
$\CL(B)$ the subgroup of $\CE(B)$ of Morita equivalences 
given by bimodules with linear source, and by $\CT(B)$ the 
subgroup of $\CL(B)$ of Morita equivalences given by bimodules
with trivial source (the fact that these are subgroups follows from
\ref{Vvarphi-comp} below). If $A$ is a source algebra of a block
with defect group $P$, we denote by $\Aut_P(A)$ the group
of algebra automorphisms of $A$ which fix the image of $P$ in
$A$ elementwise, and by $\Out_P(A)$ the quotient of 
$\Aut_P(A)$ by the subgroup of inner automorphisms induced
by conjugation with elements in $(A^P)^\times$.

Let $P$ be a finite $p$-group and $V$ an endopermutation
$\OP$-module having an indecomposable direct summand with 
vertex $P$. By results of Dade \cite{Dadeendo}, the indecomposable 
direct summands of $V$ with vertex $P$ are all isomorphic. For
any subgroup $Q$ of $P$, denote by $V_Q$ an indecomposable
direct summand of $\Res^P_Q(V)$ with vertex $Q$. The tensor
product of two indecomposable endopermutation $\OP$-modules with vertex 
$P$ has an indecomposable direct summand with vertex $P$; this 
induces an abelian group structure on the set of isomorphism
classes of indecomposable endopermutation $\OP$-modules.
The resulting group is denoted $D_\CO(P)$, called the {\it
Dade group of $P$ over $\CO$}.
 
Let $\CF$ be a saturated fusion system on $P$. We denote by
$\foc(\CF)$ the subgroup of $P$ generated by elements of the
form $\varphi(x)x^{-1}$, where $x\in$ $P$ and $\varphi\in$
$\Hom_\CF(\langle x\rangle, P)$.  Slightly modifying the terminology in  
\cite[3.3]{LiMa} we say that an endopermutation $\OP$-module $V$ having
an indecomposable direct summand with vertex $P$ is {\it $\CF$-stable} 
if for every isomorphism $\varphi : Q\to$ $R$ in $\CF$ between two 
subgroups $Q$, $R$ of $P$ we have $V_Q\cong$ ${_\varphi{V_R}}$. (In 
\cite[3.3]{LiMa} this would be the definition of $\CF$-stability for 
the class of $V$ in $D(P)$). Here ${_\varphi{V_R}}$ is the $\OQ$-module 
which is equal to $V_R$ as an $\CO$-module, with $x\in$ $Q$ acting as 
$\varphi(x)$ on $V_R$. The isomorphism classes of $\CF$-stable 
indecomposable endopermutation $\OP$-modules with vertex $P$ form a 
subgroup of $D(P)$, denoted $D(P,\CF)$. See \cite{LiMa} for more 
details. The fusion stable linear characters of $P$ form a subgroup of 
$D_\CO(P,\CF)$ which we identify with the group 
$\Hom(P/\foc(\CF),\CO^\times)$. We set 
$\Out_\CF(P)=$ $\Aut_\CF(P)/\Inn(P)$. Following the notation introduced 
in \cite[1.13]{AOV}, we denote by $\Aut(P,\CF)$ the group of 
automorphisms of $P$ which stabilise $\CF$ and  set $\Out(P,\CF)=$ 
$\Aut(P, \CF)/\Aut_\CF(P)$. The group $\Aut(P)$ acts on $D_\CO(P)$, with 
$\Inn(P)$ acting trivially. This action restricts to an action of 
$\Aut(P,\CF)$ on the subgroup $D_\CO(P,\CF)$ of $D_\CO(P)$, with 
$\Aut_\CF(P)$ acting trivially on that subgroup, and hence inducing an 
action of $\Out(P,\CF)$ on $D_\CO(P,\CF)$ and its torsion subgroup 
$D^t_\CO(P,\CF)$.  Again by \cite[7.6]{Puigbook}, Morita equivalences 
with an endopermutation source have sources which are stable with 
respect to the involved fusion systems. The assumption that the 
residue field $k$ is large enough in the statements below implies that 
the fusion systems of the blocks are saturated.

\begin{thm}  \label{thm:endoMorita}
Let $G$ be a finite group and $\CO$ be a complete local principal ideal 
domain with residue field $k$ of characteristic $p$ such that $k$ is 
large enough for all subgroups of $G$.  
Let $b$ a block of $\OG$ with defect group $P$ 
and source idempotent $i$ in $(\OGb)^P$. Set $A=$ $i\OG i$ and $B=$ 
$\OGb$. Denote by $\CF$ the fusion system on $P$ determined by $A$.

\begin{enumerate}
\item[{\rm (i)}]
Let $M$ be a $(B,B)$-bimodule which induces a Morita equivalence and 
which has an endopermutation module as a source. Then $M$ is isomorphic 
to a direct summand of 
$$\OG i\tenOP \Ind^{P\times P}_{\Delta\varphi}(V)\tenOP i\OG$$
for some $\varphi\in$ $\Aut(P,\CF)$, and some $\CF$-stable 
indecomposable endopermutation $\OP$-module $V$, regarded as an
$\CO\Delta\varphi$-module via the isomorphism $P\cong$ $\Delta\varphi$
sending $x\in$ $P$ to $(x,\varphi(x))$.

\item[{\rm (ii)}]
The correspondence sending $M$ to the pair $(V,\varphi)$ induces a 
group homomorphism $\Phi$ making the following diagram 
of groups commutative with exact rows:
$$\xymatrix{
1 \ar[r] & \Out_P(A) \ar[r] \ar@{=}[d]& \CE(B) \ar[rr]^(0.33){\Phi} & 
&  D_\CO(P,\CF) \rtimes \Out(P,\CF) \\
1 \ar[r] & \Out_P(A) \ar[r] \ar@{=}[d] & \CL(B) \ar[rr]^(0.33){\Phi} \ar[u] & 
&  \Hom(P/\foc(P),\CO^\times) \rtimes \Out(P,\CF) \ar[u] \\
1 \ar[r] & \Out_P(A) \ar[r] & \CT(B) \ar[rr]^(0.33){\Phi} \ar[u]  &
&  \Out(P,\CF) \ar[u] \\
}$$
where the upwards arrows are the inclusions. 

\item[{\rm (iii)}]
The group $\CE(B)$ is finite. Moreover, if $\Phi$ maps $\CT(B)$ onto 
$\Out(P,\CF)$, then $\Phi$ maps $\CE(B)$ to the finite group 
$D^t_\CO(P,\CF)\rtimes\Out(P,\CF)$.
\end{enumerate}
\end{thm}

This will be proved in Section \ref{endoMoritaSection}. 

\begin{rem} \label{PicRemark}
\mbox{}

\smallskip\noindent {\bf (a)}
By a result of Puig \cite[14.9]{Puigmodules}, the group $\Out_P(A)$ is 
canonically isomorphic to a subgroup of the finite abelian $p'$-group 
$\Hom(E,k^\times)$, where $E\cong$ $\Out_\CF(P)$ is the inertial 
quotient of $b$. 

\smallskip\noindent {\bf (b)}
We have a canonical embedding $\Out_P(A)\to$ $\Pic(B)$ given by
the correspondence sending $\alpha\in$ $\Aut_P(A)$ to the
$(B,B)$-bimodule $\OG i_\alpha \tenA i\OG$.
By a result due independently to Scott \cite{Scottnotes} and Puig 
\cite{Puigbook}, the image of the group $\Out_P(A)$ in $\Pic(B)$ 
under this embedding is the subgroup of $\Pic(B)$ given by invertible
bimodules which are summands of $\OG i\tenOP i\OG$; that is, invertible
trivial source bimodules with diagonal vertex $\Delta P$ arising in 
statement (i) with $V=\CO$ and $\varphi=$ $\Id_P$. See for instance 
\cite[4.1]{Lisplendid} for a proof of this fact (which we will use 
repeatedly). 

\smallskip\noindent {\bf (c)}
All groups except possibly  $D_\CO(P,\CF)$ in the above diagram are 
finite.  The image of $\CE(B)$ under $\Phi$ is not known in general, 
except in some special cases, including blocks with cyclic or Klein 
four defect groups; see the theorems \ref{endoMoritacyclic} and 
\ref{endoMoritaKlein} below. 

\smallskip\noindent {\bf (d)}
By \cite[Theorem 1.1]{Linfoc}, if $\CO$ contains a primitive $|G|$-th 
root of unity, then  we have a canonical isomorphism 
$$\CL(B) \cong \CT(B) \times \Hom(P/\foc(\CF), \CO^\times)$$
compatible with the inclusion $\CT(B)\to$ $\CL(B)$ and $\Phi$. The 
subgroup of $\CL(B)$ isomorphic to $\Hom(P/\foc(\CF),\CO^\times)$ is 
equal to the intersection of $\ker(\Pic(B)\to\Pic(k\tenO B))$ and 
$\CL(B)$, and $\Phi$ restricts to the identity on this subgroup.

\smallskip\noindent {\bf (e)}
Theorem \ref{thm:endoMorita} applies to $\CO=k$. In general, the
canonical map $\CE(B)\to$ $\CE(k\tenO B)$ is surjective (by
\cite[Theorem 1.13]{KeLichar}), and its kernel is
$\Hom(P/\foc(\CF),\CO^\times)$. 

\smallskip\noindent {\bf (f)}
If $P$ is abelian, then $\Out(P,\CF)=$ $N_{\Aut(P)}(E)/E$, where
$E=$ $\Aut_\CF(P)$. We will use this identification in the statement
of the next result.
\end{rem}

Combining results of Puig \cite{Puabelien}, Zhou \cite{ZhouYY}, 
Hertweck and Kimmerle \cite{HeKi}, and Carlson and Rouquier \cite{CaRo}  
yields the following result. Unlike Theorem \ref{thm:endoMorita}, the
remaining results in this section require $\CO$ to have characteristic
zero.

\begin{thm} \label{FrobeniusPicard}
Let $G$ be a finite group. Suppose that $\chr(\CO)=0$ and that the  
residue field $k$ is a splitting field for the subgroups of $G$.  
Let $b$ a block of $\OG$ with a nontrivial abelian defect group $P$. 
Set $B=$ $\OGb$, let $i$ be a source idempotent in $B^P$, and set 
$A=$ $iBi$. Denote by $E$ the inertial quotient of of $B$ associated
with the choice of $i$, regarded as a subgroup of $\Aut(P)$, and 
suppose that $E$ is nontrivial cyclic and acts freely on 
$P\setminus \{1\}$. Set $N_E=$ $N_{\Aut(P)}(E)/E$. 
We have $\Pic(B) = \CE(B)$, and there is  an injective 
group homomorphism  $\Psi : \Pic(B) \to$ 
$\Hom(E,k^\times) \rtimes N_E$ 
which makes the following diagram with exact rows commutative:
$$\xymatrix{1 \ar[r] & \Out_P(A) \ar[r] \ar[d] 
& \Pic(B) \ar[r]^{\Phi} \ar[d]^{\Psi} & 
D_\CO(P,\CF)\rtimes N_E \ar[d] & \\
1 \ar[r] & \Hom(E,k^\times) \ar[r] 
& \Hom(E,k^\times) \rtimes N_E \ar[r] 
& N_E \ar[r] & 1}$$ 
where the top row is from Theorem \ref{thm:endoMorita}, the bottom row 
is the canonical exact sequence, the left vertical arrow is the 
canonical embedding, and the right vertical map is the canonical 
surjection with kernel $D_\CO(P,\CF)$. 
\end{thm}

There is no loss of generality in assuming that $E$ is nontrivial: 
if $E$ is trivial, then $B$ is nilpotent, hence Morita equivalent to
$\OP$. See the Example \ref{nilpotentEx} below for more details. 
Note that the image of $\Phi$ need not be contained in the group
$N_E=$ $N_{\Aut(P)}(E)/E$ since $\CE(B)$ need not be equal to 
$\CT(B)$. Theorem \ref{FrobeniusPicard}  will be proved in Section 
\ref{FrobSection}.

For the sake of completeness, we describe the Picard groups for blocks 
with a cyclic or Klein four defect group over $\CO$ with $\chr(\CO)=0$, 
since these two cases have some additional properties. This is for the 
most part well-known and a combination of various results in the 
literature, such as \cite{RoSc87}, \cite{Weiss}, 
\cite[4.3, 5.6, 5.8]{Licyclic}, \cite[1.1]{Likleinfour}, 
\cite[1.1]{CEKL}, \cite[\S 11.4]{LiKoZi}. In particular, in both of 
these cases, if $B$ is not nilpotent, then $\Pic(B)$ is equal to 
$\CT(B)$, and hence $\Im(\Phi)$ is contained in $N_E$, where the 
notation is as in Theorem \ref{FrobeniusPicard}. 
Furthermore, if $P$ is cyclic, then $\Aut(P)$ is abelian, and 
hence the semidirect product in the last statement of Theorem 
\ref{FrobeniusPicard} becomes a direct product.

\begin{thm} \label{endoMoritacyclic}
Let $G$ be a finite group. Suppose that $\chr(\CO)=0$ and that the  
residue field $k$ is a splitting field for the subgroups of $G$.  
Let $b$ a block of $\OG$ with a nontrivial cyclic defect group $P$. 
Set $B=$ $\OGb$, let $i$ be a source idempotent in $B^P$, and set 
$A=$ $iBi$. Denote by $E$ the associated inertial quotient of $B$, and
suppose that $E$ is nontrivial. Let  $M$ be an $(B,B)$-bimodule inducing 
a Morita equivalence. Then $M$ is a trivial source module. We have 
$$\Pic(B)= \CT(B) \cong \Out_P(A) \times \Aut(P)/E\ ,$$
the group $\Out_P(A)$ is cyclic of order dividing $|E|$, with generator
the isomorphism class $\Omega^n_{B\tenO B^\op}(B)$, where $n$ is the
smallest positive integer such that this bimodule induces a Morita
equivalence.
\end{thm}

The fusion system $\CF$ on $P$ determined by $i$ in the above theorem
is equal to that of the group $P\rtimes E$. Since the automorphism
group of $P$ is abelian, it follows that $\Out(P,\CF)=$ $\Aut(P)/E$.
If $\CO$ has characteristic zero and if two blocks with cyclic defect 
groups are Morita equivalent, then there is a Morita equivalence with 
endopermutation source, and hence Theorem \ref{endoMoritacyclic} 
implies that any Morita equivalence between blocks over $\CO$ with 
cyclic defect groups has endopermutation source.  We have similar 
results for Klein four defect groups.

\begin{thm} \label{endoMoritaKlein}
Let $G$ be a finite group.
Suppose that $p=2$, that $\chr(\CO)=$ $0$, and that the residue field 
$k$ is a splitting field for the subgroups of $G$. Let $b$ a block of 
$\OG$ with a Klein four defect group $P$. Set $B=$ $\OGb$, and suppose
that $B$ is not nilpotent. Let  $M$ be a $(B,B)$-bimodule inducing a  
Morita equivalence. Then $M$ is a trivial source module. 
If $B$ is Morita equivalent to $\CO A_4$, then  
$$\Pic(B) = \CT(B) \cong S_3\ ,$$
if $B$ is Morita equivalent to the principal block of $\CO A_5$, then  
$$\Pic(B) = \CT(B) \cong C_2\ ,$$
\end{thm}

In Theorem \ref{endoMoritaKlein}, since $B$ is not nilpotent, the fusion 
system $\CF$ of $B$ on $P$ is that of the group $A_4\cong$ 
$P\rtimes C_3$, and $\Out(P,\CF)\cong$ $S_3/C_3\cong$ $C_2$. Using the 
classification of finite simple groups via \cite{CEKL}, it follows that 
every Morita equivalence between two non-nilpotent blocks over $\CO$ 
with Klein four defect groups has trivial source, and any Morita 
equivalence between two nilpotent blocks with Klein four defect groups 
has linear source (and there cannot be a Morita equivalence between a 
nilpotent and non-nilpotent block with Klein four defect groups). 
The theorems \ref{endoMoritacyclic} and \ref{endoMoritaKlein}  will be 
proved in Section \ref{FrobSection}.

\medskip
 If $\CO$ is a $p$-adic ring, 
and $B$ a block algebra of a finite group 
algebra over $\CO$, then $\Pic(B)$ is known to be  finite, 
(see \cite[Theorems (55.19), (55.25)]{CRII}). We give an alternative proof of 
this fact as part of Theorem \ref{Picfinite} below. 
If $k$ is an algebraic closure of $\Fp$ then we do not know whether $\Pic(B)$ 
is finite, but the following result shows that $\Pic(B)$ is the colimit
of the Picard groups of blocks over some $p$-adic rings; in
particular, $\Pic(B)$ is a torsion group. If $\CO$  is of characteristic zero, by a $p$-adic subring of 
$\CO$ we mean a finite extension 
$R'$ of the $p$-adic integers contained in $\CO$. 
The $p$-adic subrings of $\CO$ form a directed  system; the colimit of 
this system is the union in $\CO$ of the $p$-adic subrings. 

\begin{thm} \label{Pic-finite}
Suppose that $k$ is an algebraic closure of $\Fp$ and that $\chr(\CO)=$
$0$. Let $G$ be a finite group and $b$ a central idempotent of $\OG$. 
Set $B=$ $\OGb$. 

\smallskip\noindent (i)
Let $M$ be a $(B,B)$-bimodule inducing a Morita equivalence on $B$. 
There exists a $p$-adic subring $\CO_0$ of $\CO$ such that $b\in$ 
$\CO_0G$ and such that, setting $B_0=$ $\CO_0 Gb$, we have  
$M\cong$ $\CO\otimes_{\CO_0} M_0$ for some  $(B_0, B_0)$-bimodule 
$M_0$ inducing a Morita equivalence on $B_0$. 

\smallskip\noindent (ii)
The group $\Pic(B)$ is the colimit of the finite groups 
$\Pic(\CO_0 Gb)$, with $\CO_0$ running over the $p$-adic subrings of
$\CO$ such that $b\in$ $\CO_0 G$. In particular, all elements in 
$\Pic(B)$ have finite order. 
\end{thm}

We prove this theorem and related facts in Section 
\ref{PicfiniteSection}. It is well known that  a  result  analogous   to Theorem \ref{Pic-finite}   holds over $k$; see  Lemma \ref{PicFpbar} below.   On the  other hand, unlike the situation in characteristic zero, if  $k$ is algebraically closed  then  it is known that Picard groups   of block algebras over   $k$ are   not in  general finite.

\begin{rem}  \label{EBrem}
Two Morita equivalent blocks of finite groups via a bimodule with 
endopermutation source have isomorphic defect groups.
It is an open problem at present whether Morita equivalent blocks
always have isomorphic defect groups.
\end{rem}

\begin{rem} \label{WeissRem}
Let $P$ be a finite $p$-group and $Q$ a normal subgroup of $P$.
Suppose that $\chr(\CO)=0$. Weiss' criterion states that if $M$
is a finitely generated $\OP$-module such that $\Res^P_Q(M)$ is a free 
$\OQ$-module and such that $M^Q$ is a permutation $\CO P/Q$-module,
then $M$ is a permutation $\OP$-module. This was proved by Weiss in 
\cite[Theorem 2]{Weiss}, \cite[\S 6]{Weiss2} for $\CO$ the ring of 
$p$-adic integers $\Z_p$, extended to finite extensions of $\Z_p$ by 
Roggenkamp in \cite[Theorem II]{Roggsub}, 
and further extended to $\CO$ with perfect residue field in work of
Puig \cite[Theorem A.1.2]{Puigbook}, and for general $\CO$ in work of 
McQuarrie, Symonds, and Zalesskii \cite[Theorem 1.2]{MSZ}. 
Weiss' criterion is a key ingredient for calculating Picard groups of 
block algebras of finite groups. In applications below, we will
make use of the fact that $Q$-fixed points in the above sitation are
isomorphic to cofixed points. More precisely, with the notation and 
hypotheses above, one verifies that there is an $\CO P/Q$-module 
isomorphism $\CO\tenOQ M\cong$ $M^Q$ sending $1\ten m$ to $\Tr^Q_1(m)$, 
where $m\in$ $M$ and $\Tr^Q_1(m)=$ $\sum_{y\in Q} ym$. 
\end{rem} 

\section{Tensoring bimodules with endopermutation source}
\label{tensorSection}

Some calculations of Picard groups will involve stable equivalences
of Morita type and stable Picard groups. We briefly review these
notions. Let $B$, $C$ be symmetric $\CO$-algebras with separable
semisimple quotients. 
Following terminology introduced by Brou\'e, we say that a 
$(B,C)$-bimodule $M$ induces a stable equivalence of Morita type 
between $B$ and $C$, if $M$ is finitely generated projective as a left 
$B$-module, as a right $C$-module, and if we have bimodule isomorphisms 
$M\tenC M^*\cong $ $B\oplus X$ for some projective 
$B\tenO B^\op$-module and $M^*\tenB M\cong$ $C\oplus Y$ for some 
projective $C\tenO C^\op$-module. 

The {\it stable Picard group of $B$} is the group $\uPic(B)$ of 
isomorphism classes in the stable category of $(B,B)$-bimodules 
inducing a stable equivalence of Morita type on $B$. The group 
structure on $\uPic(B)$ is induced by taking tensor products over 
$B$. Any Morita equivalence is a stable equivalence of Morita type, 
and hence if $B$ has no nonzero projective summand as a 
$B\tenO B^\op$-module, then we have an inclusion of groups 
$\Pic(B)\subseteq$ $\uPic(B)$. 

Denoting by $\Omega(B)$ the kernel of a projective cover of $B$ as
a $B\tenO B^\op$-module, it is well-known that if $B$ has no nonzero 
projective summand as a $B\tenO B^\op$-module, then $\Omega(B)$ 
induces a stable equivalence of Morita type on $B$, and the image of 
$\Omega(B)$ in $\uPic(B)$ generates a cyclic central subgroup of 
$\uPic(B)$, which we will denote by $\langle \Omega(B) \rangle$. 
See e. g. \cite[Proposition 2.9]{Listable}. 

By \cite[Theorem 2.1]{Listable}, if $B$ and $C$ are in addition 
indecomposable as algebras and not projective as
modules over $B\tenO B^\op$ and $C\tenO C^\op$, respectively,  
and if $M$ induces a stable equivalence of Morita type between $B$ and 
$C$, then $M$ has a unique indecomposable nonprojective direct summand 
in any decomposition as a direct sum of indecomposable 
$B\tenO C^\op$-modules. 

For $P$, $Q$ finite groups and $\varphi : P\to$ $Q$ a group isomorphism 
set 
$$\Delta\varphi= \{(x,\varphi(x))\ |\ x\in P\}\ .$$  
We write $\Delta P$ instead of $\Delta\Id_P$. We regard an 
$\OP$-module $U$ as an $\CO\Delta\varphi$-module with $(x,\varphi(x))$ 
acting on $U$ as $x$, where $x\in$ $P$. That is, the action of 
$\Delta\varphi$ on $U$ is determined by the action of the first 
component of an element in $\Delta\varphi$; this accounts for the 
slight asymmetry in the statements (i) and (ii) in the next Lemma. If 
$V$ is an $\OQ$-module, we denote by ${_\varphi{V}}$ the $\OP$-module 
which is equal to $V$ as an $\CO$-module, such that $x\in$ $P$ acts as 
$\varphi(x)$ on $V$. We use similar notation for right modules and 
bimodules. We regard an $\CO(P\times Q)$-module $M$ as an 
$(\OP,\OQ)$-bimodule via $xmy^{-1}=$ $(x,y)\cdot m$, where $x\in$ $P$, 
$y\in$ $Q$ and $m\in$ $M$. 

\begin{lem} \label{Indtensor}
Let $P$, $Q$, $R$ be finite groups Let $\varphi: P\to$ $Q$ and 
$\psi : Q\to$ $R$ be group isomorphisms. Let $V$ be an $\OP$-module and 
$W$ an $\OQ$-module. 

\begin{enumerate}
\item[{\rm (i)}]
We have an $(\OP, \OQ)$-bimodule isomorphism
$$\Ind^{P\times Q}_{\Delta\varphi}(V)\cong 
(\Ind^{P\times P}_{\Delta P}(V))_{\varphi^{-1}}$$
sending $(x,y)\ten v$ to $(x,\varphi^{-1}(y))\ten v$, for all
$x\in$ $P$, $y\in$ $Q$ and $v\in$ $V$.

\item[{\rm (ii)}]
We have an $(\OP, \OQ)$-bimodule isomorphism
$$\Ind^{P\times Q}_{\Delta\varphi}(V)\cong 
{_\varphi{(\Ind^{Q\times Q}_{\Delta Q}({_{\varphi^{-1}}{V}}))}}$$
sending $(x,y)\ten v$ to $(\varphi(x), y)\ten v$, for all
$x\in$ $P$, $y\in$ $Q$ and $v\in$ $V$.

\item[{\rm (iii)}]
We have an $(\OP, \OR)$-bimodule isomorphism
$$\Ind^{P\times Q}_{\Delta\varphi}(V)\tenOQ 
\Ind^{Q\times R}_{\Delta\psi}(W)\cong 
\Ind^{P\times R}_{\Delta(\psi\circ\varphi)}(V\tenO ({_\varphi{W}}))\ .$$
\end{enumerate}
\end{lem}

\begin{proof}
The statements (i) and (ii) are straightforward verifications. Statement
(iii) is a special case of a more general result of Bouc 
\cite[Theorem 1.1]{Bouc10}. One can prove (iii) also by first showing 
this for $P=Q=R$, $\varphi=\psi=\Id_P$ (see e. g. 
\cite[Corollary 2. 4. 13]{LiBook}), and then using (i), (ii) to obtain 
the general case. 
\end{proof}

\begin{lem} \label{Indtensordual}
Let $P$, $Q$ be finite groups and $\varphi : P\to$ $Q$ a group
isomorphism, and $V$ a finitely generated $\CO$-free $\OP$-module. 
We have an $(\OP, \OP)$-bimodule isomorphism
$$\Ind^{P\times Q}_{\Delta\varphi}(V)^* \cong 
\Ind^{Q\times P}_{\Delta\varphi^{-1}}({_{\varphi^{-1}}(V^*)})\ .$$
\end{lem}

\begin{proof}
Induction and duality commute, and hence
$\Ind^{P\times Q}_{\Delta\varphi}(V)^* \cong$ 
$\Ind^{P\times Q}_{\Delta\varphi}(V^*)$ as $\CO(P\times Q)$-modules.
The isomorphism $P\times Q\cong$ $Q\times P$ exchanging the two
factors sends $(x,\varphi(x))\in$ $\Delta\varphi$ to 
$(\varphi(x),x)=$ $(\varphi(x),\varphi^{-1}(\varphi(x)))\in$ 
$\Delta\varphi^{-1}$, which accounts for the subscript $\varphi^{-1}$
in the last term in the statement of the Lemma. 
\end{proof}

\begin{lem} \label{Indzeta}
Let $P$ be a finite group and $\zeta : P\to$ $\CO^\times$ a group
homomorphism. Denote by $\CO_\zeta$ the $\OP$-module
$\CO$ with $x\in$ $P$ acting by multiplication with $\zeta(x)$.
Denote by $\tau$ the $\CO$-algebra automorphism of $\OP$
satisfying $\tau(x)=$ $\zeta^{-1}(x)x$ for all $x\in$ $P$.
There is an $\CO(P\times P)$-module isomorphism
$$\Ind^{P\times P}_{\Delta P}(\CO_\zeta)\cong \OP_\tau$$
sending $(x,y)\ten 1$ to $\zeta(y)xy^{-1}$ for $x$, $y\in$ $P$.
\end{lem}

\begin{proof}
This is a straightforward verification.
\end{proof}

For $P$, $Q$ finite $p$-groups, $\CF$ a fusion system on $P$
and $\varphi : P\to$ $Q$ a group isomorphism,  we denote by
${^\varphi{\CF}}$ the fusion system on $Q$ induced by $\CF$
via the isomorphism $\varphi$. That is, for $R$, $S$ subgroups
of $P$, we have $\Hom_{{^\varphi{\CF}}}(\varphi(R),\varphi(S))=$
$\varphi\circ\Hom_\CF(R,S)\circ\varphi^{-1}$, where we use the
same notation $\varphi$, $\varphi^{-1}$ for their restrictions to
$S$, $\varphi(R)$, respectively. 

\begin{lem}[{\cite[7.6]{Puigbook}}] \label{Vvarphi-existence}
Let $G$, $H$ be finite groups, $b$, $c$ blocks of $\OG$, $\OH$ with 
nontrivial defect groups $P$, $Q$, respectively,
and let $i\in$ $(\OGb)^P$ and $j\in$ $(\OHc)^Q$ be source  
idempotents. Denote by $\CF$ the fusion system on $P$
of $b$ determined by $i$, and denote by $\CG$ the fusion
system on $Q$ determined by $j$.  
Let $M$ be an indecomposable $(\OGb, \OHc)$-bimodule with 
endopermutation source inducing a stable equivalence of Morita type. 

Then there is an isomorphism $\varphi : P\to$ $Q$ and an
indecomposable endopermutation $\OP$-module
$V$ such that $M$ is isomorphic to a direct summand of 
$$\OG i\tenOP \Ind_{\Delta\varphi}^{P\times Q}(V) \tenOQ j\OH$$
as an $(\OGb,\OHc)$-bimodule. For any such $\varphi$ and
$V$, the following hold.

\begin{enumerate}
\item[{\rm (i)}]
$\Delta\varphi$ is a vertex of $M$ and $V$, regarded as an
$\CO\Delta\varphi$-module, is a source of $M$.

\item[{\rm (ii)}]
We have ${^\varphi{\CF}}=$ $\CG$, and the endopermutation 
$\OP$-module $V$ is $\CF$-stable.
\end{enumerate}
\end{lem}

See  \cite[9.11.2]{LiBook} for a proof of the above Lemma
using the terminology of the present paper. 

\begin{lem}[{cf. \cite{Pulocsource}}] \label{locsource}
Let $G$ be a finite group, $b$ a block of $\OG$, $P$ a defect group
of $b$, and $i\in$ $(\OGb)^P$ a source idempotent. Denote by
$\CF$ the fusion system on $P$ of the block $b$ determined by the 
choice of $i$. Let $\varphi\in$ $\Aut(P)$. The following are equivalent.

\begin{enumerate}
\item[{\rm (i)}]
We have $\varphi\in \Aut_\CF(P)$.

\item[{\rm (ii)}]
We have $i\OG \cong$ ${_\varphi(i\OG)}$ as $(\OP,\OG)$-bimodules.

\item[{\rm (iii)}]
As an $\CO(P\times P)$-module, $i\OG i$ has an indecomposable direct 
summand isomorphic to $\Ind^{P\times P}_{\Delta\varphi}(\CO)$.
\end{enumerate}
\end{lem}

\begin{proof}
The equivalence of (i) and (ii) is proved as part of 
\cite[Theorem 8.7.4]{LiBook}, and the equivalence of (i) and (ii)
is proved as part of \cite[Theorem 8.7.1]{LiBook}. 
\end{proof}

\begin{lem} \label{Vvarphi-comp}
Let $G$, $H$, $L$ be finite groups, $b$, $c$, $d$ blocks of $\OG$, 
$\OH$, $\OL$ with nontrivial defect groups $P$, $Q$, $R$, respectively.
Let $i$, $j$, $l$ be source idempotents in $(\OGb)^P$, $(\OHc)^Q$, 
$(\OLd)^R$, respectively. Let $V$ be an indecomposable endopermutation
$\OP$-module with vertex $P$ and $W$ an indecomposable endopermutation
$\OQ$-module with vertex $Q$. Let $\varphi : P\to$ $Q$ and 
$\psi : Q\to$ $R$ be group isomorphisms.

Let $M$ be an indecomposable direct summand of the 
$(\OGb,\OHc)$-bimodule
$$X = \OG i\tenOP \Ind_{\Delta\varphi}^{P\times Q}(V) \tenOQ j\OH$$
and let $N$ be an indecomposable direct summand of the
$(\OHc,\OLd)$-bimodule
$$Y = \OH j\tenOQ \Ind_{\Delta\psi}^{Q\times R}(W) \tenOR l\OL \ .$$
Suppose that $M$ induces a stable equivalence of Morita type between 
$\OGb$ and $\OHc$, and suppose that $N$ induces a stable equivalence
of Morita type between $\OHc$ and $\OLd$. Then the indecomposable
nonprojective direct summand of the $(\OGb, \CO Ld)$-bimodule 
$M\tenOHc N$ is isomorphic to an indecomposable direct summand with
vertex $\Delta(\psi\circ\varphi)$ of the $\CO(G\times L)$-module
$$\OG i\tenOP \Ind^{P\times R}_{\Delta(\psi\circ\varphi)}(U)
\tenOR l\OL$$
where $U$ is an indecomposable direct summand with vertex $P$ of 
$V\tenO {_\varphi{W}}$.
\end{lem}

\begin{proof}
Identify $P$, $Q$, $R$ via $\varphi$, $\psi$; that is, we may assume
that $\varphi$, $\psi$ are the identity maps. With this identification,
the groups $\Delta\varphi$ and $\Delta\psi$ are both equal to
$\Delta P$. 
It follows from Lemma \ref{Vvarphi-existence} that the source 
idempotents $i$, $j$, $l$ of $b$, $c$, $d$, respectively, determine
the same fusion system $\CF$ on $P$, and that $V$, $W$ are $\CF$-stable.

Clearly $M\tenOHc N$ induces a stable equivalence of Morita type, and
hence its unique (up to isomorphism) indecomposable nonprojective 
summand $Z$ induces a stable equivalence of Morita type, and moreover,
$Z$ is isomorphic to an indecomposable summand of 
$$X\tenOHc Y =$$
$$\OG i\tenOP \Ind_{\Delta P}^{P\times P}(V) \tenOP j\OH j\tenOP
 \Ind_{\Delta P}^{P\times P}(W) \tenOP l\OL$$
Since $Z$ is indecomposable, we may replace the term
$j\OH j$ in the middle by some indecomposable $(\OP,\OP)$-bimodule
summand. Since $Z$ must have a vertex of order at least 
$|P|$,  it follows from Lemma \ref{locsource} that we may choose that 
bimodule summand of $j\OH j$ to be isomorphic to 
$\Ind^{P\times P}_{\Delta\epsilon}(\CO)$ for 
some $\epsilon\in$  $\Aut_\CF(P)$. Thus $Z$ is isomorphic to a direct
summand of 
$$\OG i\tenOP \Ind_{\Delta P}^{P\times P}(V) \tenOP 
\Ind^{P\times P}_{\Delta\epsilon}(\CO) \tenOP
 \Ind_{\Delta P}^{P\times P}(W) \tenOP l\OL$$
By Lemma \ref{Indtensor}, this is isomorphic to
$$\OG i\tenOP 
\Ind_{\Delta \epsilon}^{P\times P}(V\tenO ({_\epsilon{W}})) \tenOP 
l\OL$$
The fusion stability of $W$ implies that this is isomorphic to
$$\OG i\tenOP \Ind_{\Delta\epsilon}^{P \times P}(V\tenO W) \tenOP 
l\OL$$
By Lemma \ref{Indtensor} again, this is isomorphic to 
$$\OG i\tenOP \Ind^{P\times P}_{\Delta P}(V\tenO W)_{\epsilon^{-1}} 
\tenOP l\OL$$
hence to 
$$\OG i\tenOP \Ind^{P\times P}_{\Delta P}(V\tenO W) 
\tenOP {_\epsilon{l\OL}}$$
Now $\epsilon\in$ $\Aut_\CF(P)$, and hence Lemma \ref{locsource}
implies that ${_\epsilon{l\OL}}\cong$ $l\OL$ as $(\OP,\OL)$-bimodules.
Together it follows that $Z$ is isomorphic to a direct summand of
$$\OG i\tenOP \Ind^{P\times P}_{\Delta P}(V\tenO W) 
\tenOP l\OL $$
Since $V$, $W$ are indecomposable endopermutation $\OP$-modules, it
follows that $U$ is up to isomorphism the unique indecomposable direct
summand with vertex $P$ of $V\tenO W$. The fact that $Z$ has a vertex 
of order $|P|$ implies that $Z$ is isomorphic to a direct summand of 
$$\OG i\tenOP \Ind^{P\times P}_{\Delta P}(U) \tenOP l\OL $$
as claimed.
\end{proof} 

The vertex-source pairs of an indecomposable module over a finite 
group algebra are unique up to conjugation. For the situation in Lemma 
\ref{Vvarphi-existence}, this translates to the following statement.

\begin{lem} \label{Vvarphi-uniqueness}
Let $G$, $H$ be finite groups, $b$, $c$ blocks of $\OG$, $\OH$ with 
nontrivial defect groups $P$, $Q$, respectively, and let $i\in$ 
$(\OGb)^P$ and $j\in$ $(\OHc)^Q$ be source idempotents. Denote by $\CF$ 
the fusion system on $P$ of $b$ determined by $i$, and denote by $\CG$ 
the fusion system on $Q$ determined by $j$. Let $M$ be an indecomposable
$(\OGb, \OHc)$-bimodule inducing a stable
equivalence of Morita type. Let $\varphi$, $\psi : P\cong$ $Q$ be 
group isomorphisms, let $V$, $W$ be indecomposable endopermutation
$\OP$-modules.
Suppose that $M$ is isomorphic to a direct summand of both 
$$X = \OG i\tenOP \Ind_{\Delta\varphi}^{P\times Q}(V) \tenOQ j\OH$$
and
$$Y = \OG i\tenOP \Ind_{\Delta\psi}^{P\times Q}(W) \tenOQ j\OH\ .$$
Then $\varphi^{-1}\circ\psi\in$ $\Aut_\CF(P)$ and 
$W\cong V$ as $\OP$-modules.
\end{lem}

\begin{proof}
Since $\OG$, $\OH$ are symmetric algebras, we have $(\OG i)^*\cong$ 
$i\OG$ as $(\OP,\OG)$-bimodules and $(j\OH)^*\cong$ $\OH j$ as
$(\OH,\OQ)$-bimodules. Since the terms in $Y$ are all finitely
generated projective as one-sided modules, duality anticommutes with
tensor products. We therefore have 
$$Y^*\cong 
\OH j \tenOQ \Ind^{Q\times P}_{\Delta \psi^{-1}}({_{\psi^{-1}}W^*})
\tenOP i\OG\ .$$
By Lemma \ref{Vvarphi-existence} we have ${^\varphi{\CF}}=$ $\CG=$
${^\psi{\CF}}$, and the modules $V$, $W$ are $\CF$-stable. Then
$W^*$ is clearly $\CF$-stable as well. 
By the assumptions, $\OGb$ is a summand of the $(\OGb, \OGb)$-bimodule 
$M\ten_\OHc M^*$. Thus $\OGb$ is  isomorphic to a direct summand of the 
bimodule
$$X\ten_\OHc Y^*\ .$$
By Lemma \ref{Indtensor} and Lemma \ref{Vvarphi-comp}, it follows that 
$\OGb$ is isomorphic to a direct summand of 
$$\OG i \tenOP \Ind^{P\times P}_{\Delta(\psi^{-1}\circ\varphi)}(U)
\tenOP i\OG$$
where $U$ is an indecomposable direct summand of 
$$V\tenO ({_{\psi^{-1}\circ\varphi}(W^*)})$$
with vertex $P$. As an $\CO(G\times G)$-module, $\OGb$ has 
$\Delta P$ as a vertex and a trivial source (for any vertex), and hence 
$U=\CO$. Thus the $\CO(P\times P)$-module $i\OG i$, and hence also
the $\CO(P\times P)$-module
$$i\OG i \tenOP \Ind^{P\times P}_{\Delta(\psi^{-1}\circ\varphi)}(\CO)
\tenOP i\OG i$$
has an indecomposable direct $\CO(P\times P)$-summand isomorphic
to $\Ind^{P\times P}_{\Delta P}(\CO)$. Lemma \ref{locsource} implies
that there are $\alpha$, $\beta\in$ $\Aut_\CF(P)$ such that 
$$\Ind^{P\times P}_{\Delta P}(\CO)\cong  
\Ind^{P\times P}_{\Delta\alpha}(\CO)\tenOP
\Ind^{P\times P}_{\Delta(\psi^{-1}\circ\varphi)}(\CO) \tenOP
\Ind^{P\times P}_{\Delta\beta}(\CO)$$
and hence by Lemma \ref{Indtensor} we have an isomorphism of
$\CO(P\times P)$-modules
$$\Ind^{P\times P}_{\Delta P}(\CO)\cong  
\Ind^{P\times P}_{\Delta \tau}(\CO)$$
where
$$\tau = \beta\circ\psi^{-1}\circ\varphi\circ\alpha\ .$$
Thus $\Delta P$ and $\Delta\tau$ are both vertices of this
$\CO(P\times P)$-module, hence they are conjugate in $P\times P$,
and this implies that $\tau$ is an inner automorphism of $P$. 
But then $\tau$, $\alpha$, $\beta$ are all in $\Aut_\CF(P)$, and hence
so is $\psi^{-1}\circ\varphi=$ $\beta^{-1}\circ\tau\circ\alpha^{-1}$.

Finally, the fusion stability of $W$ implies that 
$$V\tenO ({_{\psi^{-1}\circ\varphi}(W^*)})\cong V\tenO W^*$$
and this module has by the above a trivial summand, forcing $W\cong$ 
$V$.
\end{proof}

Given a block algebra $B$ of a finite group algebra over $\CO$, 
we denote by $\uCE(B)$, $\uCL(B)$, $\uCT(B)$ the subgroups of 
$\uPic(B)$ of isomorphism classes of bimodules inducing a stable
equivalence of Morita type, having endopermutation, linear, trivial
source, respectively. The canonical group homomorphism $\Pic(B)\to$
$\uPic(B)$ sends the subgroups $\CE(B)$, $\CL(B)$, $\CT(B)$ to
$\uCE(B)$, $\uCL(B)$, $\uCT(B)$, respectively.

\begin{lem} \label{Vvarphi-conj}
Let $G$, $H$, be finite groups, $B$, $C$, block algebras of $\OG$, 
$\OH$, with a common nontrivial defect group $P$. Let $i$, $j$, be 
source idempotents in $B^P$, $C^P$, respectively. Suppose that the $B$ 
and $C$ have the same fusion system $\CF$ on $P$ with respect to the 
choice of  $i$ and $j$.  Let $V$ be an indecomposable endopermutation 
$\OP$-module with vertex $P$, and let $M$ be an indecomposable 
$(B,C)$-bimodule summand of 
$$\OG i\tenOP \Ind^{P\times P}_{\Delta P}(V)\tenOP j\OH$$
such that $M$ induces a Morita equivalence (resp. a stable equivalence 
of Morita type) between $B$ and $C$. The following hold.

\begin{enumerate}
\item[{\rm (i)}]
The functor $M^*\ten_B - \ten_B M$ induces a group isomorphism 
$\CE(B)\cong$ $\CE(C)$ (resp. a group isomorphism $\uCE(B)\cong$ 
$\uCE(C)$). 

\item[{\rm (ii)}]
Suppose that $V\cong$ $\Omega_P^n(\CO)$ for some integer $n$.
The functor $M^*\ten_B - \ten_B M$ induces group isomorphisms
$\CL(B)\cong$ $\CL(C)$ (resp. $\uCL(B)\cong$ $\uCL(C)$) and
$\CT(B)\cong$ $\CT(C)$ (resp. $\uCT(B)\cong$ $\uCT(C)$). 
\end{enumerate}
\end{lem}

\begin{proof}
Since $M$ and its dual $M^*$ induce a Morita equivalence, it follows
that the functor $M^*\ten_B - \ten_B M$ induces a group isomorphism
$\Pic(B)\cong$ $\Pic(C)$. It follows from the hypotheses on $M$ and from 
Lemma \ref{Vvarphi-comp} that this isomorphism preserves the subgroups 
determined by Morita equivalences with endopermutation sources, and
hence induces a group isomorphism $\CE(B)\cong$ $\CE(C)$. 
If $V$ is a Heller translate of $\CO$, then $V$ is stable under
any automorphism of $P$, and hence by the Lemmas \ref{Indzeta}
and \ref{Vvarphi-comp}, the isomorphism $\CE(B)\cong$ $\CE(C)$ 
restricts to isomorphisms $\CL(B)\cong$ $\CL(C)$ and $\CT(B)\cong$ 
$\CT(C)$. The argument for stable equivalences of Morita type is 
strictly analogous.
\end{proof}

\begin{rem}
With the notation of Lemma \ref{Vvarphi-conj}, if $V$ is not stable 
under $\Aut(P)$, then the functor  $M^*\ten_B - \ten_B M$ 
need not induce isomorphisms $\CL(B)\cong \CL(C)$ or 
$\CT(B)\cong \CT(C)$; see the examples \ref{nilpotentEx} and 
\ref{productExa} below.
\end{rem}

\section{Proof of Theorem \ref{thm:endoMorita}}
\label{endoMoritaSection}

We use the notation and hypotheses as in the statement of Theorem 
\ref{thm:endoMorita}. If $P=$ $\{1\}$, then $B$ is a matrix algebra
over $\CO$. Thus $\Pic(B)$ is trivial, hence all groups in
in the diagram in (ii) are trivial, and the result is clear in that
case. Assume that $P$ is nontrivial. The fact that $M$ is isomorphic 
to a summand of 
$$\OG i\tenOP \Ind^{P\times P}_{\Delta\varphi}(V)\tenOP i\OG$$
with $V$ and $\varphi$ as stated in (i)  follows from Puig's result
stated as Lemma \ref{Vvarphi-existence} above. This proves (i).

Lemma \ref{Vvarphi-uniqueness} implies that there is 
a well-defined map $\Phi$ from $\CE(\OGb)$ to 
$D_\CO(\CF,P)\rtimes \Out(P,\CF)$ as stated. Lemma \ref{Vvarphi-comp} 
shows that this map is a group homomorphism.
The kernel of $\Phi$ consists of all isomorphism classes of 
bimodules $M$ such that the corresponding pair $(V,\varphi)$ 
satisfies  $\varphi\in$ $\Aut_\CF(P)$ and $V\cong$ $\CO$. 
Since ${_\varphi{i\OG}}\cong$ $i\OG$ as $(\OP,\OG)$-bimodule whenever
$\varphi\in$ $\Aut_\CF(P)$, it follows that we may chose $\varphi=$ 
$\Id_P$. Note that $\Ind^{P\times P}_{\Delta P}(\CO)\cong$ $\OP$ as
$(\OP,\OP)$-bimodules. Thus the isomorphism classes of bimodules  $M$ in 
the kernel of $\Phi$ correspond to trivial source bimodule summands of 
$$\OG i \tenOP i\OG\ .$$
As noted in the Remark \ref{PicRemark} (b), these in turn correspond 
bijectively to the elements of $\Out_P(A)$, via the map $\Aut_P(A)\to$ 
$\CT(B)$ sending $\alpha\in$ $\Aut_P(A)$ to the isomorphism class of the 
$(B,B)$-bimodule $\OG i_\alpha \tenA i \OG$. 
Any bimodule of this form induces a Morita equivalence on $B$ because 
$\OG i$ induces a Morita equivalence between $B$ and $A$. Moreover, $A$ 
is isomorphic to a direct summand of $A\tenOP A$ 
(cf. \cite[Theorem 6.4.7]{LiBook}), hence the previous 
bimodule is isomorphic to a direct summand of $\OG i\tenOP i\OG$, where 
we make use of the fact that $\alpha$ restricts to the identity on the 
image of $\OP$ in $A$. This shows the exactness of the first row of
the diagram in the statement. Since by the previous
argument, the image of $\Aut_P(A)$ in $\CE(B)$ is contained 
in $\CT(B)$, the remaining rows of the diagram are exact as well.
This proves (ii). 

The finiteness of $\CE(B)$ follows from combining two facts:
first, a source of a bimodule $M$ inducing a Morita equivalence
on $B$ has an $\CO$-rank bounded in terms of $B$, and second,
endopermutation modules over a fixed $p$-group over a field of 
characteristic $p$ are defined over a fixed finite field and lift to 
the ring of Witt vectors of that finite field (this is a consequence 
of the classification of endopermutation modules; see 
\cite[Theorem 13.2, Theorem 14.2]{Thevguided}). 
This implies that only finitely many isomorphism classes of 
indecomposable endopermutation modules arise as sources of 
bimodules inducing Morita equivalences of $B$. Since there are only 
finitely many isomorphism classes of indecomposable modules with a 
fixed vertex-source pair, the finiteness of $\CE(B)$ follows. 
Alternatively, again using the above mentioned facts on endopermutation 
modules, one obtains the finiteness of $\CE(B)$ as a consequence of the 
finiteness of $\Pic(B)$ whenever $\CO$ is a $p$-adic ring (cf. 
\cite[Theorems (55.19, (55.25)]{CRII}).

Suppose that $\Phi$ maps  $\CT(B)$ onto $\Out(P,\CF)$.  
Let $M$ be an indecomposable direct summand of 
$$\OG i\tenOP \Ind^{P\times P}_{\Delta\varphi}(V)\tenOP i\OG$$
inducing a Morita equivalence,
for some $\varphi\in$ $\Out(P,\CF)$ and some indecomposable
endopermutation $\CO\Delta\varphi$-module $V$ (which is then
$\CF$-stable, when regarded as an $\OP$-module by 
\ref{Vvarphi-existence}).  Since $\Phi$ is assumed to map
$\CT(B)$ onto $\Out(P,\CF)$, there is an indecomposable
direct bimodule summand $N$ of
$$\OG i\tenOP \Ind^{P\times P}_{\Delta\varphi}(\CO)\tenOP i\OG$$
inducing a Morita equivalence. Since $\Phi$ is a group homomorphism,
it follows that $U=M\tenB N^*$ is a bimodule inducing a Morita 
equivalence with vertex $\Delta P$ and source $V$, regarded as an
$\OP$-module. That is, the image of the class of $U$ under $\Phi$ is
the class of the pair $(V,\Id_P)$, hence 
contained in the subgroup $D_\CO(P,\CF)$. In other words, $\Phi$ sends
the class of $U$ to the class of $V$ in $D_\CO(P,\CF)$.  By
the above, the group $\CE(B)$ is finite, and hence the image of $V$ 
in $D_\CO(P,\CF)$ has finite order. This proves (iii), which 
concludes the proof of Theorem \ref{thm:endoMorita}.

\section{Picard groups of blocks with abelian defect and cyclic
Frobenius inertial quotient} 
\label{FrobSection}

In this section we make the blanket assumption that the residue field 
$k$ is large enough for the finite groups and their subgroups which
arise here. This hypothesis is in particular needed for the Theorems 
\ref{FrobeniusPicard}, \ref{endoMoritacyclic}, \ref{endoMoritaKlein},
whose proofs are given at the end of this section.
We need the following result of Puig.

\begin{thm} [{\cite[6.8]{Puabelien}}]
\label{frobeniusstableequivalence}
\index{Frobenius inertial quotient}
Let $G$ be a finite group, $B$ a block of $\OG$, $P$ a defect group
of $b$, $i\in$ $B^{P}$ a source idempotent of $B$ and $e$ the block 
idempotent of $kC_G(P)$ satisfying $\Br_{P}(i)e\neq$ $0$. Suppose that 
$P$ is nontrivial abelian. Set $E=$ $N_G(P,e)/C_G(P)$ and suppose that 
$E$ acts freely on $P\setminus\{1\}$. Set $L=$ $P\rtimes E$. Then there 
is an indecomposable $E$-stable endopermutation $\OP$-module $V$ with 
vertex $P$ and an indecomposable $(B, \OL)$-bimodule $M$ with the
following properties.

\smallskip\noindent (i)
The bimodule $M$ is a direct summand of
$\OG i\tenOP\Ind^{P\times P}_{\Delta P}(V) \tenOP\OL$.

\smallskip\noindent (ii)
As an $\CO(G\times L)$-module, $M$ has vertex $\Delta P$
and source $V$.

\smallskip\noindent (ii)
The bimodule $M$ and its dual $M^*$ induce a stable equivalence
of Morita type between $B$ and $\OL$.
\end{thm}

For a proof of this theorem using the notation and terminology above, 
see \cite[Theorem 10.5.1]{LiBook}. Note that in the situation of
Theorem \ref{frobeniusstableequivalence}, if the acting group $E$ is
abelian, then $E$ is in fact cyclic (cf. \cite[Theorem 10.3.1]{Gor}). 
The next result we need is due to Carlson and Rouquier.

\begin{pro}[{\cite[Corollary 3.3]{CaRo}}] \label{uPicFrobenius}
Let $P$ be a nontrivial abelian $p$-group and $E$ an abelian 
$p'$-subgroup of $\Aut(P)$ acting freely on $P\setminus \{1\}$.
Set $L=$ $P\rtimes E$. Denote by $\Omega$ the Heller operator on the 
category of $(\OL, \OL)$-bimodules.
Any indecomposable $(\OL, \OL)$-bimodule inducing a
stable equivalence of Morita type on $\OL$ is isomorphic to 
$\Omega^n(M)$ for some integer $n$ and some $(\OL, \OL)$-bimodule
$M$ inducing a Morita equivalence on $\OL$. Equivalently, we have 
$$\uPic(\OL) = \Pic(\OL)\cdot \langle \Omega(\OL) \rangle\ ,$$
where we identify $\Omega(\OL)$ with its image in $\uPic(\OL)$.
\end{pro}

As mentioned at the beginning of Section \ref{tensorSection}, the cyclic 
subgroup $\langle \Omega(\OL) \rangle$ is in the center of 
$\uPic(\OL)$. We need to determine the structure of $\Pic(\OL)$. The 
following result is a slight refinement of results of Y. Zhou 
\cite[Theorem 14]{ZhouYY} and Hertweck and Kimmerle 
\cite[Theorem 4.6]{HeKi} in the special case where the focal subgroup
$[P,E]$ of $P$ in $P\rtimes E$ is equal to $P$. The proof follows in 
part that of \cite[Lemma 4.2]{Licyclic}. As in the aforementioned papers 
\cite{ZhouYY}, \cite{HeKi},  \cite{Licyclic}, the key step is an 
application of Weiss' criterion.

\begin{pro}
\label{Piclocal}
Let $P$ be a nontrivial abelian $p$-group and $E$ an abelian 
$p'$-subgroup of $\Aut(P)$ such that $[P,E]=P$. Set $L=$ $P\rtimes E$.
Denote by $\CF$ the fusion system of $L$ on $P$. Assume that
$\chr(\CO)=0$. Any class in $\Out(\OL)$ has a representative in 
$\Aut(\OL)$  which stabilises the image of $P$ in $\OL$ as a set, and 
any $(\OL, \OL)$-bimodule which induces a Morita equivalence has 
trivial source. We have canonical group isomorphisms
\begin{align*}
\Pic(\OL) & = \CT(\OL) \\
           & \cong \Out(\OL) \\
           & \cong \Out_P(\OL)\rtimes N_{\Aut(P)}(E)/E\ , \\ 
\Out_P(\OL)& \cong\Hom(E,k^\times)\ ,\\
\Out(P,\CF) & \cong N_{\Aut(P)}(E)/E \ .
\end{align*} 
Moreover, the inverse image of  $N_{\Aut(P)}(E)/E$ in $\Out(\OL)$
consists of all classes of automorphisms of $\OL$ which stabilise the 
trivial $\OL$-module. 
\end{pro}

\begin{proof} 
Since $E$ is abelian, it follows that the algebra $\OL$ is basic, and
hence the canonical map $\Out(\OL)\to$ $\Pic(\OL)$ is a group 
isomorphism. Since $p'$-roots of unity in $k^\times$ lift uniquely to
roots of unity of the same order in $\CO^\times$, it follows that the 
canonical map $\CO^\times\to$ $k^\times$ induces a group isomorphism
$\Hom(E,\CO^\times)\to$ $\Hom(E,k^\times)$. Any group homomorphism
$\zeta : E\to$ $\CO^\times$ yields an algebra automorphism
$\eta$ of $\OPE$ defined by $\eta(uy)=$ $\zeta(y)uy$ for all $u\in P$ 
and all $y\in E$; in particular, $\eta\in$ $\Aut_P(\OL)$. If $\zeta$ is 
nontrivial, then the $\OL$-module ${_\eta{\CO}}$ is nontrivial, and 
hence $\eta$ is not inner. Thus the correspondence $\zeta\to\eta$ 
induces an injective group homomorphism $\Hom(E,k^\times)\to$ 
$\Out_P(\OL)$, and by \cite[14.9]{Puigmodules} this is an isomorphism. 
Note that the above arguments show that nontrivial classes in 
$\Out_P(\OL)$ do not stabilise the trivial $\OL$-module; we will make 
use of this fact later. 

If $\alpha$ is any algebra automorphism of $\OL$, then ${_\alpha\CO}$ 
is a module of rank one, hence corresponds to a group homomorphism 
$\zeta : E\to$ $\CO^\times$. Denoting by $\eta$ the algebra homomorphism 
as above, it follows that $\eta^{-1}\circ\alpha$ stabilises the trivial 
$\OPE$-module. Thus $\alpha=$ $\eta\circ(\eta^{-1}\circ\alpha)$ is a
product of an automorphism in $\Aut_P(\OL)$ and an automorphism which 
stabilises the trivial module.

It remains to show that the outer automorphism group of $\OL$ of 
automorphisms which stabilise the trivial module is canonically 
isomorphic to $N_{\Aut(P)}(E)/E$. If $\psi$ is an automorphism of $P$ 
which normalises $E$, then by elementary group theory, $\psi$ extends to 
a group automorphism of $P\rtimes E$, hence to an algebra automorphism 
$\beta$ of $\OL$ which stabilises the trivial module. Moreover,
$\beta$ normalises $\Out_P(\OL)$. The group of outer automorphisms
obtained in this way intersects $\Out_P(\OL)$ trivially, as remarked
earlier.
 
By a result of Coleman \cite{Coleman} (or by results of Puig on fusion
in block source algebras in \cite[Theorem 3.1]{Pulocsource}), 
any inner algebra automorphism of $\OPE$ which stabilises $P$
acts on $P$ as some automorphism in $E$. Therefore $\beta$ is inner if 
and only if $\psi\in$ $E$, so the above correspondence yields an 
injective group homomorphism $N_{\Aut(P)}(E)/E\to$ $\Out(\OL)$, and the
image of this homomorphism in $\Out(\OL)$ normalises $\Out_P(\OL)$
and intersects $\Out_P(\OL)$ trivially.  

We need to show that the image of this group homomorphism consists of 
all classes of automorphisms which stabilise the trivial module, and 
this is done using Weiss' criterion (see Remark \ref{WeissRem} above).  

Since $[P,E]=P$, it follows that no nontrivial linear character of $P$
extends to $L$. Thus every simple $kL$-module lifts
to a unique irreducible character, and the characters arising as lifts
of simple $kL$-modules are precisely the characters having $P$ in their
kernel. Again since $E$ is abelian, the left $kL$-module $kE$ (with 
$P$ acting trivially) is a direct sum of a set of representatives of
the isomorphism classes of simple $kL$-modules, and hence the
$\OL$-module $\CO E$, with $P$ acting trivially, is the unique lift (up
to isomorphism) of $kE$ to an $\CO$-free $\OL$-module. Thus the 
isomorphism class of $\CO E$ is stable under any algebra automorphism of 
$\OL$. Note that $\CO E\cong$ $\OL\tenOP\CO$. Let $\alpha\in$ 
$\Out(\OL)$. Then the $(\OP,\OP)$-bimodule ${_\alpha{\OL}}$ is free as 
a right $\OP$-module, and we have isomorphisms 
${_\alpha{\OL}}\tenOP\CO\cong$ ${_\alpha{\CO E}}\cong$ $\CO E$. 
This is a permutation module as a left 
$\OP$-module (since $P$ acts in fact trivially on this module). Weiss' 
criterion implies that ${_\alpha{\OL}}$ is a permutation 
$\CO(P\times P)$-module. As an $\CO(P\times L)$-module, 
$\OL$ is indecomposable, and hence ${_\alpha{\OL}}$ is indecomposable 
as well. Thus ${_\alpha{\OL}}$ is a trivial source 
$\CO(P\times L)$-module, with a `twisted diagonal' vertex of 
the form $\Delta\varphi$ for some automorphism $\varphi$ of $P$. Thus 
${_\alpha{\OL}}$ is isomorphic to a direct summand of 
$\Ind^{P\times L}_{\Delta\varphi}(\CO)\cong$ 
${_\varphi{\OL}}$, where the last isomorphism sends $(u,x)\ten 1$ to
$\varphi(u)x^{-1}$ for $u\in P$ and $x\in L$. 
Comparing ranks yields an isomorphism ${_\alpha{\OL}}\cong$ 
${_\varphi{\OL}}$ as $(\OP,\OL)$-bimodules. Any such isomorphism is in 
particular an isomorphism of right $\OL$-modules, hence is induced by 
left multiplication with an element $c$ in $\OL^\times$. The fact that 
this is also a homomorphism of left $\OP$-modules implies that 
$c\alpha(u)=$ $\varphi(u)c$ for all $u\in $ $P$. Thus after replacing 
$\alpha$ by its conjugate by $c$, we may assume that $\alpha$ extends 
$\varphi$. Let $y\in E$, and denote by $c_y$ the inner automorphism of 
$\OPE$ given by conjugation with $y$. Then 
$\alpha\circ c_y \circ\alpha^{-1}$ acts on $P$ as 
$\varphi\circ c_y \circ \varphi^{-1}$. But 
$\alpha\circ c_y \circ\alpha^{-1} = $ $c_{\alpha(y)}$ acts on $P$ also 
as the automorphism given by conjugation with $\alpha(y)$. 
As before, any inner algebra automorphism of $\OL$ which stabilises 
$P$ acts on $P$ as some automorphism in $E$. Thus $c_{\alpha(y)}$ acts
on $P$ as some element in $E$, or equivalently, $\varphi$  normalises 
$E$. It follows from the above that $\varphi$ extends to a group 
automorphism of $L$, which in turn extends to an algebra
automorphism $\beta$, whose class is by construction in the image
of the map $N_{\Aut(P)}(E)/E\to$ $\Out(P)$. Since $\alpha$ and $\beta$
induce the same action as $\varphi$, it follows that 
$\beta^{-1}\circ\alpha\in$ $\Aut_P(\OL)$. Therefore, if $\alpha$ 
stabilise the trivial $\OL$-module, so does $\beta^{-1}\circ\alpha$, and
hence $\beta^{-1}\circ\alpha$ is inner. Equivalently, $\alpha$ and
$\beta$ have in that case the same image in $\Out(\OL)$. The result 
follows.
\end{proof}

This Proposition shows that the last sequence in  Theorem 
\ref{thm:endoMorita} applied to $B=$ $\OPE$ as above is a split short
exact sequence. The hypothesis $[P,E]=P$ is in particular satisfied
if $E$ is nontrivial and acts freely on $P\setminus \{1\}$. 

\begin{cor} \label{uPicT}
Let $P$ be a nontrivial abelian $p$-group and $E$ a nontrivial cyclic 
$p'$-subgroup of $\Aut(P)$ acting freely on $P\setminus \{1\}$.
Set $L=$ $P\rtimes E$. Assume that $\chr(\CO)=0$. 
Let $M$ be an $(\OL, \OL)$-bimodule inducing a stable equivalence
of Morita type. Then $M$ has an endotrivial source as an
$\CO(L \times L)$-module. In particular, we have
$\uPic(\OL)=\uCE(\OL)$.
\end{cor}

\begin{proof}
If $M$ induces a Morita equivalence, then $M$ has a trivial source by
Proposition \ref{Piclocal}. Thus the Heller translates of $M$ have
as source Heller translates of the trivial module (for a vertex),
and these are endotrivial. The result follows.
\end{proof}

\begin{proof}[Proof of Theorem \ref{FrobeniusPicard}]
We use the notation of Theorem \ref{FrobeniusPicard}, and we set
$L=$ $P\rtimes E$. With the notation of Theorem 
\ref{frobeniusstableequivalence}, since $M$ induces a stable equivalence 
of Morita type, it follows that the functor $M^*\tenB - \tenB M$ induces 
a group isomorphism 
$\uPic(B)\cong$ $\uPic(\OL)$. By Corollary \ref{uPicT}, we have
$\uPic(\OL)=$ $\uCE(\OL)$. Lemma \ref{Vvarphi-conj} implies that 
$\uPic(B)=$ $\uCE(B)$; in particular, we have $\Pic(B)=$ $\CE(B)$. 
The above isomorphism $\uPic(B)\cong$ $\uPic(\OL)$ restricts to  
an injective group homomorphism $\Pic(B)\to$ $\uPic(\OL)$.
Since $\Pic(B)=$ $\CE(B)$, the group $\Pic(B)$ is a finite subgroup
of $\uPic(\OL)=$ $\Pic(\OL)\cdot \langle\Omega(\OL) \rangle$, where
the last equality uses Proposition \ref{uPicFrobenius}. 
We show next that the image of the group homomorphism
$$\Pic(B) \to \uPic(\OL)$$
induced by $M^*\tenB -\tenB M$ is contained in $\Pic(\OL)$. 

If $P$ is not cyclic, then $\langle \Omega(\OL) \rangle$ is an infinite
cyclic central subgroup of $\uPic(\OL)$. Thus any finite subgroup of
$\uPic(\OL)$ is contained in $\Pic(\OL)$, which has the structure
as stated by Proposition \ref{Piclocal}. Thus if $P$ is not cyclic,
then the functor $M^*\tenB - \tenB M$ induces an injective group
homomorphism $\Pic(B)\to$ $\Pic(\OL)$.

If $P$ is cyclic, this is true as well, but requires a slightly 
different argument since in that case $\OL$ is periodic as a bimodule 
(see e. g. \cite[Lemma 4.1]{Licyclic}), and  hence $\uPic(\OL)$ is 
already a finite group. If $P$ is cyclic, then the assumption $E\neq$ 
$\{1\}$ implies that $|P|\geq$ $3$. The algebra $kL$ is a symmetric 
Nakayama algebra, and thus even powers of the Heller operator on $kL$ 
permute the isomorphism classes of simple modules. It follows that 
every even power of the Heller operator on $\OL$ is induced by a Morita 
equivalence. Hence, in order to show that the $(\OL,\OL)$-bimodule
$M^*\tenB N\tenB M$ induces a Morita equivalence, we may replace $N$ by 
any even Heller translate as a $(B,B)$-bimodule. Corollary 
\ref{OmegaMoritaCor} below implies we may assume that $N\tenB-$ 
stabilises the isomorphism classes of all finitely generated 
$k\tenO B$-modules. It follows that the stable equivalence of Morita 
type given by $M^*\tenB N\tenB M$ stabilises the isomorphism classes of 
all finitely generated $kL$-modules, hence is a  Morita equivalence by 
\cite[Theorem 2.1]{Listable}. 

Back to the general case, it follows from Remark \ref{PicRemark} (b) 
that the canonical image of $\Out_P(A)$ in $\Pic(B)$ is mapped to a 
subgroup of $\Hom(E,k^\times)$ under the injective group homomorphism 
$\Pic(B)\to$ $\Pic(\OL)$. This shows the commutativity of the left 
rectangle in the statement of Theorem \ref{FrobeniusPicard}. For the 
commutativity of the right rectangle in that diagram, let $N$ be a 
$(B,B)$-bimodule inducing a Morita equivalence on $B$. Let $U$ be an 
indecomposable endopermutation $\OP$-module with vertex $P$ and 
$\varphi\in$ $N_{\Aut(P)}(E)$ such that the image of $[N]$ in 
$D_\CO(P,\CF)\rtimes N_E$ under the map $\Phi$ is equal to 
$[U][\varphi]$, where $[U]$ is the class of $U$ in the Dade group and 
$[\varphi]$ is the image of $\varphi$ in $N_E$. The map $\Psi$ sends 
$[N]$ to the class of the $(\OL, \OL)$-bimodule $N'=$ 
$M^*\tenB N\tenB M$. By Lemma \ref{Vvarphi-comp} the image of $[N']$ 
in $D_\CO(P,\CF)\rtimes N_E$ is
$$[V^*]\cdot [U]\cdot [\varphi]\cdot [V] = 
[V^*\tenO U\tenO {^\varphi{V}}] \cdot [\varphi] =[\varphi]$$
where the last equality uses the fact that $\Pic(\OL)=$ $\CT(\OL)$
from Proposition \ref{Piclocal}. This shows the commutativity of the
right rectangle in the diagram and concludes the proof of Theorem 
\ref{FrobeniusPicard}. 
\end{proof}

Let $G$ be a finite group and $b$ a block of $kG$ with a nontrivial 
cyclic defect group $P$ and inertial quotient $E$.  We collect some 
well-known results on cyclic blocks, going back to Brauer, Dade, and 
Green. Our notation follows \cite{Licyclic} (where proofs and further 
references can be found).
 
Set $B=$ $kGb$, and denote by $I$ a set of representatives of the 
conjugacy classes of primitive idempotents in $B$. We have $\ell(B)=$ 
$|E|=$ $|I|$. For any $i\in$ $I$, the projective indecomposable 
$B$-module $Bi$ has uniserial submodules $U_i$, $V_i$ such that 
$U_i+V_i=$ $J(B)i$ and $U_i\cap V_i=$ $\soc(Bi)$. One may choose
notation such that $\Omega(U_i)\cong$ $V_{\rho(i)}$ for some 
$\rho(i)\in$ $I$ and $\Omega(V_i)\cong$ $U_{\sigma(i)}$ for some 
$\sigma(i)\in$ $I$. The maps $\rho$, $\sigma$ defined in this way are 
then permutations of $I$, and $\rho\circ\sigma$ is a transitive cycle 
on $I$. The Brauer tree of $B$ is defined as follows. Each vertex 
corresponds to exactly one $\rho$-orbit, or one $\sigma$-orbit. The 
edges are labelled by the elements in $I$, in such a way that the edge 
$i$ connects the $\rho$-orbit $i^\rho$ of $i$ with the $\sigma$-orbit 
$i^\sigma$ of $i$.  The permutations $\rho$ and $\sigma$ induce cyclic 
permutations of the set of edges emanating from a fixed vertex. If 
$|P|\geq$ $3$, then the $2 |E|$ modules $U_i$, $V_i$, $i\in I$, are 
pairwise nonisomorphic, because they correspond to the simple 
$k(P\rtimes E)$-modules and their Heller translates through a stable 
equivalence of Morita type  between $B$ and $k(P\rtimes E)$.  In 
particular, if $|P|\geq$ $3$, then the modules $U_i$, $V_i$ all have 
period $2 |E|$, for any even integer $n$ and any $i\in$ $U_i$ there is 
$j\in$ $I$ such that $\Omega^n(U_i)\cong$ $U_j$, and for any odd  
integer $n$  and any $i\in$ $I$ there is $j\in I$ such that
$\Omega^n(U_i)\cong$ $V_j$. 

\begin{pro} \label{OmegaMorita}
Let $G$ be a finite group and $b$ a block of $kG$ with a cyclic defect 
group $P$ of order at least $3$. Set $B=$ $kGb$. Let $n$ be an integer 
such that the $(B,B)$-bimodule $M=$ $\Omega^n_{B\tenk B^\op}(B)$ induces
a Morita equivalence. Then $n$ is even. As a $k(G\times G)$-module, $M$ 
has vertex $\Delta P=$ $\{(u,u)\ |\ u\in P\}$ and trivial source. 
\end{pro}

\begin{proof}
Since the functor $M\tenB-$ is an equivalence, it permutes the
isomorphism classes of simple modules in such a way that the induced 
permutation of the edges of the Brauer tree is a tree automorphism. By 
\cite[7.2]{KeLichar} this automorphism stabilises a vertex (this is 
where we use that $|P|\geq 3$). Thus there is a $\rho$-orbit or
a $\sigma$-orbit which is stabilised by this automorphism.  Suppose
that the $\rho$-orbit $i^\rho$ of an edge $i$ is stabilised. Thus 
$M\tenB U_i\cong$ $U_{i'}$ for some $i'$ belonging to the $\rho$-orbit 
of $i$. Note that $M\tenB U_i\cong$ $\Omega^n(U_i)$. By the above,
we have $U_{i'}\cong$ $\Omega^m(U_i)$ for some even integer $m$.
Thus $\Omega^{n-m}(U_i)\cong$ $U_i$.  Since $U_i$ has period
$2 |E|$, it follows that $2 |E|$ divides $n-m$; in particular, $n-m$ is
even. Since also $m$ is even, so is $n$. A similar argument 
applies if the Brauer tree automorphism stabilises a $\sigma$-orbit.
The $k(G\times G)$-module $B$ has vertex $\Delta P$ and trivial
source. Since the trivial $kP$-module $k$ has  period $2$ and since
$n$ is even, it  follows that $M$ has vertex $\Delta P$ and trivial 
source. 
\end{proof}

\begin{cor} \label{OmegaMoritaCor}
Let $G$ be a finite group and $b$ a block of $\OG$ with a nontrivial 
cyclic  defect group $P$. Set $B=$ $\OGb$. Let $N$ be a $(B,B)$-bimodule
inducing a Morita equivalence. Then there exists an integer $n$ such 
that the bimodule $N'=$ $\Omega^n_{B\tenO B^\op}(N)$ induces a Morita 
equivalence with the property that the functor $N'\tenB -$ stabilises 
the isomorphism class of every finitely generated $k\tenO B$-module. 
Moreover, if $|P|\geq$ $3$, then any integer $n$ with this property is 
even.
\end{cor}

\begin{proof}
It follows from \cite[5.1]{Licyclic} that there is an integer $n$ such 
that $N'=$ $\Omega^n_{B\tenO B^\op}(N)$ induces a Morita equivalence 
with the property that the functor $N'\tenB -$ stabilises the 
isomorphism class of every finitely generated $k\tenO B$-module. 
Then the bimodule $N'\tenB N^*\cong$ $\Omega^n_{B\tenO B^\op}(B)$ 
induces a Morita equivalence as well. Therefore, if $|P|\geq$ $3$, 
then Proposition \ref{OmegaMorita} implies that $n$ is even.
\end{proof}

\begin{proof}[{Proof of Theorem \ref{endoMoritacyclic}}]
Let $G$ be a finite group and $b$ a block of $\OG$ with a nontrivial 
cyclic  defect group $P$ and nontrivial inertial quotient $E$.  
Set $B=$ $\OGb$. Since $E$ is nontrivial, the block $B$ is not 
nilpotent, and hence $|P|\geq$ $3$. 
Let $M$ be a $(B, B)$-bimodule inducing a Morita equivalence.  
Since $\Pic(B)=$ $\CE(B)$, the vertices of $M$ are twisted diagonal
subgroups of $P\times P$ isomorphic to $P$.
We need to show that $M$ has a trivial source, for some vertex. 
Since the trivial $\OP$-module has period $2$, it follows that $M$
has a trivial source if and only $\Omega^n_{B\tenO B^\op}(M)$ has a
trivial source for some (and then necessarily any) even integer $n$. 
Therefore, by Corollary \ref{OmegaMoritaCor} we may assume that 
the functor $M\tenB-$ stabilises the isomorphism classes of all 
finitely generated $k\tenO B$-modules.  
By \cite[4.3, 5.6]{Licyclic}, the subgroup of $\Pic(B)$ of bimodules
with this property is canonically isomorphic to $\Aut(P)/E$, and by
\cite[5.8]{Licyclic}, the bimodules with this property correspond to 
algebra automorphisms of $A$ extending group automorphisms of
$P$, and hence these bimodules have trivial source. 
By the Remark \ref{PicRemark}, the group $\Out_P(A)$ corresponds to
the subgroup of $\Pic(B)$ of isomorphism classes of bimodules 
of the form $\Omega^n_{B\tenO B^\op}(B)$, where $n$ is an integer
such that this bimodule induces a Morita equivalence. Note that in 
that case we have $\Omega^n_{B\tenO B^\op}(M)\cong$ 
$\Omega^n_{B\tenO B^\op}(B)\tenB M \cong$ 
$M\tenB \Omega^n_{B\tenO B^\op}(B)$. Thus $\Out_P(A)$ is indeed
a direct factor of $\Pic(B)$.
\end{proof}


\begin{proof}[Proof of Theorem \ref{endoMoritaKlein}]
Let $G$ be a finite group and $b$ a non-nilpotent block of $\OG$ with a 
Klein four defect group $P$. Set $B=$ $\OGb$. Every endopermutation
$\OP$-module is a Heller translate of a rank one module. 
By \cite[Theorem 1.1]{Likleinfour}, $B$ is Morita equivalent to either 
$\CO A_4$ or the principal block algebra of $\CO A_5$, via a Morita 
equivalence with source a Heller translate of the trivial module.
(Using the classification of finite simple groups 
it is shown in \cite{CEKL} that there is even a Morita equivalence with 
trivial source in these cases, but this is not needed for the present
proof). By Lemma \ref{Vvarphi-conj}, we may assume that $B$ is equal 
to one of these two algebras. Note that then $B$ is its own
source algebra.

If $B=$ $\CO A_4$, the result follows from Proposition \ref{Piclocal};
indeed, we have $\Out_P(B)\cong$ $C_3$, and $\Aut(P)\cong S_3$, 
hence $\Aut(P)/E\cong$ $S_3/C_3\cong$ $C_2$, which acts nontrivially on
$\Out_P(B)$ and hence yields $\Pic(B)\cong$ $S_3$.

Suppose that $B$ is the principal block algebra $\CO A_5 b_0$. 
As a very special case of Theorem \ref{frobeniusstableequivalence},  
identifying $\CO A_4$ to its image in this algebra, induction and 
restriction yields a splendid stable equivalence of 
Morita type between $B$ and $\CO A_4$ (this is well-known and
easy to verify directly). That is, we are in a situation of Lemma
\ref{Vvarphi-conj} in which $V=\CO$. By Theorem \ref{FrobeniusPicard}
we have an embedding $\Pic(B)\to$ $\Pic(\CO A_4)\cong$ $C_3\rtimes C_2$
which maps $\Out_P(B)$ to the subgroup $C_3$.  

We show next that $\Out_P(B)$ is trivial. Write $A_4=$ $P\rtimes E$
with $E\cong$ $C_3$. We need to show that the nontrivial 
automorphisms of the form $uy\mapsto \zeta(y)uy$ of $\CO A_4$ (with
$\zeta : E\to$ $k^\times$ a group homomorphism, $u\in P$, $y\in E$ as 
above) do not extend to $B$. Indeed, if they did, then the three 
trivial source $B$-modules corresponding (through the stable equivalence 
between $\CO A_4$ and $B$) to the three linear characters of $\CO A_4$  
would have to have the same rank because they would be transitively 
permuted by this automorphism group. This is not possible: the trivial 
character of $\CO A_4$ corresponds to the trivial character of $B$, 
while the two nontrivial characters of $\CO A_4$ correspond to 
$B$-modules of $\CO$-rank greater than $1$ (because $A_5$ has no 
nontrivial character of rank $1$). Thus $\Out_P(B)$ is trivial.

It follows that $\Pic(B)$ embeds into $C_2$. An automorphism of $B$ 
given by conjugation with an involution in $S_5$ yields a nontrivial 
element in $\Pic(B)$, whence $\Pic(B)\cong$ $C_2$.
\end{proof}

\section{Proof of Theorem \ref{Pic-finite}} 
\label{PicfiniteSection}

We assume in this Section that $\CO$ has characteristic zero.
For the description of $\Pic(B)$ as
colimits of finite groups in Theorem \ref{Pic-finite}, we will need
the following lemmas. One of the key ingredients is Brauer's Lemma 1
from \cite{Brauer41}, implying that if $K$ is a field of 
characteristic $0$ containing all $p'$-order roots of unity in an
algebraic closure, then the Schur indices over $K$ of absolutely 
irreducible characters of finite groups are all equal to $1$.

\begin{lem}  \label{centralidem} 
Let $K$ be a field  of characteristic $0$ 
containing all roots of unity of order prime to $p$ in an algebraic 
closure of $K$ and let $K_0$ be a subfield of $K$ containing all
algebraic elements in $K$. 
Let $G$ be a finite group and $X$ a finite-dimensional $KG$-module.
Then there is a $K_0G$-module $X_0$ such that $X\cong$ $K\ten_{K_0} X_0$.
\end{lem}

\begin{proof}
Since $KG$ is semisimple, we may assume that $X$ is a simple 
$KG$-module. 
Let $K'$ be a splitting field for $X$ containing $K$. Let $Y'$ be a 
simple $K'G$-module which is isomorphic to a direct summand of 
$K'\ten_K X$, and denote by $\psi$ the character of $Y'$. The field 
$K(\psi)$ generated by the values $\psi(g)$, $g\in G$, is contained in 
$K(\zeta)$ for some root of unity of $p$-power order $\zeta$, since
$K$ contains all roots of unity of order prime to $p$. The minimal
polynomial of $\zeta$ over $K$ has algebraic numbers as coefficients,
hence is also the minimal polynomial of $\zeta$ over $K_0$. Thus
the restriction of field automorphisms of $K(\zeta)$ to $K_0(\zeta)$
induces an isomorphism of Galois groups $\Gal(K(\zeta):K)\cong$
$\Gal(K_0(\zeta):K_0)$, hence an isomorphism of Galois groups
$\Gal(K(\psi):K)\cong$ $\Gal(K_0(\psi):K_0)$.
By \cite[Lemma 1]{Brauer41}, $K_0(\psi)$ is a splitting field for
$\psi$; that is, $Y'\cong$ $K'\ten_{K_0(\psi)} Y_0$ for some simple
$K_0(\psi)G$-module $Y_0$; in particular, $Y=$ 
$K(\psi)\ten_{K_0(\psi)} Y_0$ is a simple $K(\psi)$-module which appears
in $K(\psi)\tenK X$. It follows from \cite[Ch. 3 Theorem (1.30)]{NaTs}
that we have an isomorphism
$$K(\psi)\ten_K X \cong \bigoplus_{\sigma\in\Gal(K(\psi):K)}\ 
{^\sigma{Y}}$$
Combining the above isomorphisms yields that the $K_0G$-module
$$X_0 = \bigoplus_{\sigma\in\Gal(K_0(\psi):K_0)}\ {^\sigma{Y_0}}$$
satisfies 
$$K(\psi) \ten_{K_0} X_0 \cong K(\psi)\ten_K X$$
and hence $K\ten_{K_0} X_0\cong$ $X$.
\end{proof}

The following Lemma is well-known; we sketch a proof for convenience.
As before, by a $p$-adic subring of $\CO$ we mean a finite extension 
$\CO_0$ of the $p$-adic integers contained in $\CO$. Note that
$J(\CO_0)\subseteq$ $J(\CO)$. 

\begin{lem}\label{noncomplete-ramified}
Suppose that $k$ is an algebraic closure of $\Fp$. Denote by $K$ the
field of fractions of $\CO$. Let $R$ be the union of all $p$-adic 
subrings of $\CO$, and let $E$ be the field of fractions of $R$ 
identified to its image in $K$. Denote by
$\nu : K^\times \to$ $\Z$ the $\pi$-adic valuation. 

\begin{enumerate}

\item
The field $K$ is the completion of $E$ and $\CO$ is the completion of 
$R$ with respect to the topology induced by $\nu$.

\item
The restriction $\nu|_E$ of $\nu$ to $E^\times$  is a discrete valuation
with valuation ring $R$ and valuation ideal $R\cap \pi\CO$.

\item
The field $E$ is the algebraic closure of $\Q_p $ in $K$.

\end{enumerate}
\end{lem}

\begin{proof} 
Let  $W(k) \subseteq \CO $ be the ring of Witt vectors of $k$ in $\CO$. 
By the  structure  theory of complete discrete  valuation rings, 
$\CO =  W(k) [\pi]$ for an element $\pi$ satisfying a monic polynomial
$f(x)$ in $ W(k)[x]$ (see \cite[Chapter 2, Theorems 3, 4 and Chapter 1, 
Proposition 18]{SeLF}). Every element of  $W(k) $ is a limit of some 
sequence $(\sum_{0\leq i \leq  n}  \zeta_i  p^i)_{n \geq 0} $,  where
$\zeta_n \in  W(k)^{\times} $ is a $p'$-root of unity, for all 
$n \in {\mathbb N}$ (see proof of Theorem 3, Chapter 2 of \cite{SeLF}).
Since every $p'$-root of unity in $W(k)$ lies in a finite extension of
${\mathbb Z}_p$ contained in $W(k)$, it follows that $W(k)$ is the  
completion of the union of the finite extensions of ${\mathbb Z}_p$ 
contained in $W(k)$. Hence, by an application of Krasner's lemma, we 
may assume that $f(x)\in A[x]$, for some finite extension $A$ of 
${\mathbb Z}_p $ contained in $W(k)$  (see \cite[Chapter 2, 
Exercises 1,2]{SeLF}). Consequently, $\pi$ lies in a finite extension 
of $A$ and hence in $R$. Since every element of $\CO$ is a polynomial 
in $\pi$ of degree at most the degree of $f$ and with coefficients in 
$W(k)$, it follows by the same reasoning as above that every element  
of $\CO$ is a limit of a sequence of elements of $R$  implying (1).

As shown above,  $\pi \in R$. Thus   the restriction of $\nu$ to 
$E^\times$  has  valuation  group  ${\mathbb Z}$  and  $\nu|_E$ is a 
discrete valuation. In order to show (2) we need to observe that
$R= E\cap \CO$. The inclusion $\subseteq$ is trivial. For the
reverse inclusion, any element in $E\cap \CO$ belongs to the fraction
field $L$ of some $p$-adic subring $\CO_0$ of $\CO$ with nonnegative 
valuation, hence to the  valuation ring of $L$, and that is precisely 
$\CO_0$. 

If $L$ is a subfield of $K$ containing $\Q_p$ such that the degree
of $L/\Q_p$ is finite, then the valuation ring of $L$ is a $p$-adic
subring of $\CO$, and hence $L$ is contained in $E$. Since $E$ is
the union of finite extensions of $\Q_p$, all elements in $E$ are
algebraic over $\Q_p$, whence statement (3).
\end{proof}

\begin{pro}  \label{PicPro}
Suppose that $k$ is an algebraic closure of $\Fp$. Denote by $K$ the
field of fractions of $\CO$. Let $R$ be the union of all $p$-adic 
subrings of $\CO$, and let $E$ be the field of fractions of $R$ 
identified to its image in $K$. 
Let $G$ be a finite group, $b$ a central idempotent of $\OG$, and set
$B=$ $\OGb$. Let $M$ be a $(B,B)$-bimodule inducing a Morita equivalence
on $B$. There exist a $p$-adic subring $\CO_0$ of $\CO$ such that 
$b\in $ $\CO_0G$ and an $(\CO_0Gb,\CO_0Gb)$-bimodule inducing a Morita
equivalence on $\CO_0Gb$ such that $M\cong$ $\CO\ten_{\CO_0} M_0$. 
\end{pro}

\begin{proof} 
Let $X= K\otimes_{\CO} M$. So, $M$ is a full $\CO(G\times G)$-lattice 
in $X$. By Lemma  \ref{noncomplete-ramified}, 
$E$ contains all algebraic numbers in $K$. Since $k$ is an algebraic
closure of $\Fp$, it follows that $E$ contains all $p'$-order roots of
unity in an algebraic closure of $K$. Hence by Lemma 
\ref{centralidem}, there exists an $E(G\times G)$-module $X_E$ such 
that $X = K\otimes_E X_E$. Let $M_R=$ $X_E \cap M$. By Lemma 
\ref{noncomplete-ramified}, $\CO$ is the completion of  $R$ and $K$ is 
the completion of $E$.   Hence, by \cite[Corollary 30.10]{CRI}, we 
have that $M \cong \CO\otimes_{R} M_R $. Let ${\mathcal  B}$ be 
an $R$-basis of $M_R$. For $u\in G\times G $ and $x$, $y \in 
{\mathcal B}$, let $\alpha_{x,y}^u $ denote the  coefficient of    
$y$ when $ux$ is written as an $R$-linear combination of elements  
of ${\mathcal B}$.   Let  
$b = \sum_{g\in G}  \beta_g  g$, $\beta _g \in  \CO $. Then    
$\beta_g  \in R$ for all  $g \in G$. Since every finite extension of 
${\mathbb Z}_p$  in $\CO$  is contained in  a $p$-adic  subring of 
$\CO$,  there exists a $p$-adic subring  $\CO_0$ of $\CO$    
containing $\alpha_{x,y}^u$ and containing $\beta_g$ for all  
$u\in G\times G$, $x,y \in {\mathcal B}$ and all $g \in G$.     
Let $M_0$ be the $\CO_0$-submodule of $M_R$ generated by 
${\mathcal B}$. Then $M_0$ is an $(\CO_0Gb,\CO_0Gb)$-bimodule, finitely 
generated and free as $\CO_0$-module, satisfying $M_R\cong$
$R\ten_{\CO_0} M_0$. It follows that $M\cong$ $\CO\ten_{\CO_0} M_0$.
By \cite[Lemma~4.4, Prop. 4.5]{KeLichar}  we have that $M_0$ induces  
a Morita equivalence on $\CO_0Gb$ . The result follows. 
\end{proof}

\begin{proof}[{Proof of Theorem \ref{Pic-finite}}]
Statement (i) follows from Proposition \ref{PicPro}.
Thus $\Pic(B)$ is the colimit of the groups $\Pic(\CO_0 Gb)$ as
stated. By \cite[Theorems (55.19), (55.25)]{CRII}  (see also Theorem 
\ref{Picfinite} below for an alternative proof) the groups 
$\Pic(\CO_0 Gb)$ are finite, whence (ii). 
\end{proof}

\section{On the finiteness and structure of Picard groups over 
$p$-adic rings} 

In this section we provide an alternative proof for the finiteness of 
$\Pic(B)$ for a block algebra $B$ of a group algebra over a $p$-adic 
ring. We start out assuming that $\CO$ has characteristic
zero, and specialize later to the case that $\CO$ is a $p$-adic ring. 
Write $J(\CO)=\pi\CO$ for some prime element $\pi$ of $\CO$.
Let $A$ be an $\CO$-algebra which is free of finite rank as an
$\CO$-module. We use without further comment some standard facts 
relating automorphisms to Picard groups (see \cite[\S 55 A]{CRII} or
\cite[2.8.16]{LiBook} for more details). If $\alpha\in$ $\Aut(A)$, then 
the $(A,A)$-bimodule $A_\alpha$ induces a Morita  equivalence on 
$\mod(A)$, and we have $A_\alpha\cong$ $A$ as $(A,A)$-bimodules if and 
only if $\alpha$ is inner. The map sending $\alpha$ to $A_\alpha$ 
induces an injective group homomorphism $\Out(A)\to$ $\Pic(A)$. If $A$ 
is basic, this is an isomorphism (see e. g. \cite[4.9.7]{LiBook}). 
In general, if $M$ is an $(A,A)$-bimodule inducing a Morita equivalence,
then $M\cong$ $A_\alpha$ as $(A,A)$-bimodules for some $\alpha\in$ 
$\Aut(A)$ if and only if $M\cong$ $A$ as left $A$-modules. In 
particular, if the functor $M\tenA-$ stabilises the isomorphism classes 
of all simple modules, hence of all finitely generated projective 
modules, then $M\cong$ $A_\alpha$ for some $\alpha\in$ $\Aut(A)$. It 
follows that the image of $\Out(A)$ in $\Pic(A)$ has finite index 
because it contains the subgroup of $\Pic(A)$ of Morita equivalences 
which stabilises the isomorphism classes of all simple modules. 

Let $r$ be a positive integer. We denote by $\Aut_r(A)$ the 
set of $\CO$-algebra automorphisms of $A$ which induce the identity on 
$A/\pi^rA$. A trivial verification shows that $\Aut_r(A)$ is a subgroup
of $\Aut(A)$. We denote by $\Out_r(A)$ the image of $\Aut_r(A)$ in 
$\Out(A)$. The map sending $\alpha\in$ $\Aut_r(A)$ to the 
$(A,A)$-bimodule $A_\alpha$ induces a group isomorphism 
$$\Out_r(A) \cong  \ker(\Pic(A)\to\Pic(A/\pi^rA))\ ;$$ 
see e.~g.~\cite[3.1]{Linder}. In particular, $\Out_1(A)$ is isomorphic 
to the kernel of the canonical map $\Pic(A)\to$ $\Pic(k\tenO A)$, and
$\Out_r(A)$ is a normal subgroup of $\Out_1(A)$ for all $r\geq$ $1$.  

\medskip
If $\CO$ is a $p$-adic ring, $G$ a finite group and $B$ a block algebra 
of $\OG$, then by a result of Hertweck and Kimmerle \cite[3.13]{HeKi},
the group $\Out_r(B)$ is trivial, for $r$ large enough. The proof of
\cite[3.13]{HeKi} uses a theorem of Weiss in \cite{Weiss2} (restated
as Theorem 3.2 in \cite{HeKi}). As a consequence of Maranda's theorem 
\cite[(30.14)]{CRI}, this remains true if $\CO$ is not necessarily 
$p$-adic, but Maranda's theorem leads to a larger bound for $r$. More
precisely, the smallest positive integer $r$ for which $\Out_r(B)$ is 
trivial in \cite[3.13]{HeKi} depends only on the ring $\CO$, while the 
lowest bound obtained from Maranda's theorem, applied to the 
$\CO(G\times G)$-modules $B$ and $B_\alpha$, with $\alpha\in$ 
$\Out_r(B)$, would be an integer $r$ such that $\pi^r\CO=$ 
$\pi |G|^2\CO$. We give an elementary proof of the fact that $\Out_r(B)$
is trivial for $r$ large enough which does not require Weiss' 
results and which slightly improves on Maranda's bound; that is, we 
obtain a lower bound for $r$ which depends on $\CO$ and the size of a 
defect group of $B$. 

\begin{pro} \label{OutrBtrivial}
Let $G$ be a finite group, $B$ a block of $\OG$ and $P$ a defect group
of $B$. Let $d$ be the positive integer such that $\pi^d\CO=$ $|P|\CO$.
For any integer $r>d$ the group $\Out_r(B)$ is trivial.
\end{pro}

\begin{proof}
It suffices to show this for $r=d+1$. Let $\alpha\in$ $\Aut_r(B)$. In 
order to show that $\alpha$ is inner, we need to show that the 
$(B,B)$-bimodule $B_\alpha$ is isomorphic to $B$. Equivalently, we need 
to show that $B_\alpha$ and $B$ are isomorphic as 
$\CO(G\times G)$-modules. By the assumptions on $\alpha$, we have a 
bimodule isomorphism $B_\alpha/\pi^r B_\alpha\cong$ $B/\pi^r B$; in 
particular, $B_\alpha/\pi B_\alpha\cong$ $B/\pi B$ as 
$k(G\times G)$-modules. Since $p$-permutation modules over $k$ lift 
uniquely, up to isomorphism, to $p$-permutation modules over $\CO$, it 
suffices to show that $B_\alpha$ is a $p$-permutation 
$\CO(G\times G)$-module. Since $B$ is relatively $P\times P$-projective,
it suffices to show that $B_\alpha$ is a permutation 
$\CO(P\times P)$-module.

The hypotheses on $\alpha$ imply that for any $x\in$ $B$ we have 
$\alpha(x)=$ $x + \pi |P| c_x$, for some $c_x\in$ $B$. Identify the 
defect group $P$ to its image in $B$. Set $v=$ 
$\sum_{x\in P}\ \alpha(x)x^{-1}$. We have 
$$\alpha(x)x^{-1}= 1 + \pi |P| c_x x^{-1}$$ 
for all $x\in$ $P$. Thus 
$$v = |P|\cdot 1 + \pi |P| a$$ 
for some $a\in$ $B$. Set $u=$ $1 + \pi a$; this element is clearly 
invertible in $B$ and satisfies $|P|u=$ $v$. We are going to show that 
$u y u^{-1}=$ $\alpha(y)$ for all $y\in$ $P$. This is equivalent to 
showing that $uy=$ $\alpha(y)u$, hence to $vy=$ $\alpha(y)v$. In order
to show this, fix $y\in$ $P$. Then 
$$vy=\sum_{x\in P} \alpha(x)x^{-1}y=
\sum_{x\in P} \alpha(yy^{-1}x)x^{-1}y=
\sum_{x\in P} \alpha(y)\alpha(y^{-1}x)x^{-1}y\ .$$
If $x$ runs over all elements of $P$, then so does $y^{-1}x$, and 
therefore this element is equal to $\alpha(y)v$ as required. This 
implies that composing $\alpha$ with the inner automorphism given by 
conjugation with $u^{-1}$ yields an automorphism $\beta$ which is the
identity on $P$ and which belongs to the same class as $\alpha$ in
$\Out(B)$. The fact that $u\in$ $1+\pi B$ implies that $\beta\in$
$\Aut_1(B)$. Since the images of $\alpha$ and $\beta$ in $\Out(B)$
are equal, it follows that $B_\alpha\cong$ $B_\beta$ as 
$(B,B)$-bimodules, or equivalently, as $\CO(G\times G)$-modules.
Since $\beta$ is the identity on 
$P$, it follows that $B_\beta$, and therefore also $B_\alpha$, 
is a permutation $\CO(P\times P)$-module 
(isomorphic to $B$ as an $\CO(P\times P)$-module).  The result follows.
\end{proof}

Combining Proposition \ref{OutrBtrivial} with \cite[3.5]{Linder} yields 
a proof of the following result, including the finiteness of $\Pic(B)$ in
Part~(ii) in the case that $\CO$ is a $p$-adic ring.

\begin{thm}  \label{Picfinite}
Let $G$ be a finite group and $B$ a block of $\OG$ with a defect group 
$|P|$. Set $\bar B=$ $k\tenO B$. We have a canonical exact sequence of 
groups
$$\xymatrix{1 \ar[r] & \Out_1(B) \ar[r] & \Pic(B) \ar[r] &
\Pic(\bar B)}$$
with the following properties.

\smallskip\noindent (i) 
The group $\Out_1(B)$ has a finite $p$-power exponent dividing $p^d$, 
where $d$ is the positive integer satisfying $\pi^d\CO=$ $|P|\CO$.

\smallskip\noindent (ii) 
If $k$ is finite, then $\Pic(B)$ is finite, and $\Out_1(B)$ is a finite 
$p$-group. 

\smallskip\noindent (iii) 
If $k$ is an algebraic closure of $\Fp$, then every element in 
$\Pic(B)$ has finite order.
\end{thm}

We restate the part of \cite[3.5]{Linder} required for 
the proof of Theorem \ref{Picfinite}.

\begin{pro}[{cf. \cite[3.5]{Linder}}] \label{OutProp}
Let $A$ be an $\CO$-algebra which is finitely generated free as an 
$\CO$-module. Let $r$ be a positive integer. Suppose that the canonical 
map $Z(A)\to$ $Z(A/\pi^rA)$ is surjective. We have an exact sequence of 
groups
$$\xymatrix{1 \ar[r] & \Out_{2r}(A) \ar[r] & \Out_r(A) \ar[r] &
HH^1(A/\pi^r A)} \ .$$
For any integer $a\geq$ $0$, the map $\Out_r(A) \to HH^1(A/\pi^rA)$ in 
this sequence sends the subgroup $\Out_{r+a}(A)$ of $\Out_r(A)$ to 
$\pi^a HH^1(A/\pi^r)$.  
\end{pro}

We briefly review the construction of the map from $\Out_r(A)\to$ 
$HH^1(A/\pi^rA)$ in Proposition \ref{OutProp}. Let $\alpha\in$ 
$\Aut_r(A)$. For all $a\in$ $A$ we have $\alpha(a)=$ $a + \pi^r\tau(a)$ 
for some $\tau(a)\in$ $A$. By comparing the expressions $\alpha(ab)$ and
$\alpha(a)\alpha(b)$ for any two $a$, $b\in$ $A$, one sees that the 
linear map $\tau$ on $A$ induces a derivation $\bar\tau$  on $A/\pi^rA$.
(For the ring $k[[t]]$ instead of $\CO$ and $r=1$, this construction is 
due to Gerstenhaber; the derivations which arise in this way are called 
{\it integrable}.) It is shown in 
\cite[3.5]{Linder} that the map $\alpha\mapsto\bar\tau$ sends inner 
automorphisms to inner derivations, hence induces a map $\Out_r(A)\to$ 
$HH^1(A/\pi^rA)$. This is shown to be a group homomorphism. By 
construction, this map sends $\Out_{r+a}(A)$ to $\pi^a HH^1(A/\pi^r A)$,
hence has $\Out_{2r}(A)$ in its kernel. It is further shown in 
\cite[3.5]{Linder} that $\Out_{2r}(A)$ is indeed equal to that kernel.  

\medskip
Block algebras satisfy the hypothesis on the surjectivity of the 
canonical map $Z(A)\to$ $Z(A/\pi^rA)$.  We note the following  
consequences of Proposition \ref{OutProp}.

\begin{cor} \label{OutCor}
Let $A$ be an $\CO$-algebra which is finitely generated free as an 
$\CO$-module. Let $r$ be a positive integer. Suppose that the canonical 
map $Z(A)\to$ $Z(A/\pi^rA)$ is surjective. The following hold.

\smallskip\noindent (i)
The quotient $\Out_r(A)/\Out_{2r}(A)$ is abelian of exponent dividing 
$p^r$. If $k$ is finite, then $\Out_r(A)/\Out_{2r}(A)$ is a finite
abelian $p$-group of exponent dividing $p^r$.

\smallskip\noindent (ii)
The quotient $\Out_r(A)/\Out_{r+1}(A)$ is abelian of exponent dividing 
$p$. If $k$ is finite, then $\Out_r(A)/\Out_{r+1}(A)$ is a finite 
elementary abelian $p$-group. 
\end{cor}

\begin{proof}
The exact sequence in \ref{OutProp} implies that the quotient
$\Out_r(A)/\Out_{2r}(A)$ is isomorphic to a subgroup of the additive
abelian group $HH^1(A/\pi^rA)$. This group is annihilated by $\pi^r$, 
hence by $p^r$. If $k$ is finite, then $A/\pi^rA$ is a finite set, and 
hence $HH^1(A/\pi^rA)$ is a finite abelian $p$-group of exponent
dividing $p^r$. This proves (i). It follows from (i) that the quotient 
$\Out_r(A)/\Out_{r+1}(A)$ is abelian of exponent dividing $p^r$. We 
need to show that its exponent is at most $p$. Let $\alpha\in$ 
$\Aut_r(A)$. For $a\in$ $A$, write $\alpha(a)=$ $a + \pi^r\tau(a)$ for 
some $\tau(a)\in$ $A$. An easy induction shows that for $n\geq$ $1$ we 
have $\alpha^n(a)\equiv$ $a+n\pi^r\tau(a)$ modulo $\pi^{r+1}A$, and 
thus $\alpha^p\in$ $\Aut_{r+1}(A)$. The rest of (ii) follows as in (i).
\end{proof}

One can improve the statement on the exponent in statement (i) of 
Corollary \ref{OutCor} by taking ramification into account; we 
ignored this in the above proof and simply argued that if $\pi^r$ 
annihilates an $\CO$-module, then the underlying abelian group has 
exponent dividing $p^r$.
The following statement is well-known; we include a proof for 
convenience.

\begin{lem} \label{PicFpbar}
Let $k$ be an algebraic closure of $\Fp$, let $A$ be a 
finite-dimensional $k$-algebra, and let $X$ be a $k$-basis of $A$. 
Then $\Pic(A)$ is the colimit of the
finite groups $\Pic(A_\F)$, with $\F$ running over the finite subfields
of $k$ which contain the multiplicative constants of $X$, where $A_\F$ 
is the $\F$-algebra spanned by $X$. In particular, every element in 
$\Pic(A)$ has finite order.
\end{lem}

\begin{proof}
We may assume that $A$ is basic, and hence that $\Out(A)\cong$ 
$\Pic(A)$. Since $k$ is an algebraic closure of $\Fp$, all elements
in $k$ are algebraic, and hence any finite subset of $k$ generates a
finite subfield. In particular, the multiplicative constants of the 
basis $X$ generate a finite subfield $k_0$ of $A$. Let $\alpha\in$ 
$\Aut(A)$. For any $x\in$ $X$, write $\alpha(x)=$ 
$\sum_{y\in X}\ \lambda(x,y) y$ for some coefficients $\lambda(x,y)$ in 
$k$. By the above, these coefficients are contained in a finite subfield 
$k_1$ of $k$, which we may choose to contain $k_0$. Thus $\alpha$ is the 
extension to $A$ of a $k_1$-algebra automorphism $\alpha_1$ of the 
$k_1$-subalgebra $A_1$ of  $A$ with $k_1$-basis $X$, and hence
$\Pic(A)$ is a colimit as stated. Since $A_1$ is a finite 
set, it follows that $\alpha_1$ has finite order, implying that $\alpha$ 
has finite  order. This completes the proof.
\end{proof}

\begin{proof}[{Proof of Theorem \ref{Picfinite}}]
Set $\bar B=$ $k\tenO B$. As mentioned above, the kernel of the 
canonical map $\Pic(B)\to$ $\Pic(\bar B)$ is isomorphic to $\Out_1(B)$,
whence the canonical exact sequence as stated.

By \ref{OutCor}, the group $\Out_1(B)$ is filtered by the normal 
subgroups $\Out_i(B)$ such that the quotient of each two consecutive 
groups has exponent at most $p$. By Maranda's theorem 
\cite[(30.14)]{CRI} or by Proposition \ref{OutrBtrivial} above, there 
exists a positive integer $r$ such that $\Out_r(B)$ is trivial, and 
then $\Out_1(B)$ has exponent dividing $p^{r-1}$. The statement on the 
exponent in (i) follows from Proposition \ref{OutrBtrivial}. 

By \ref{OutCor} again, if $k$ is finite, then the quotients of 
subsequent subgroups in this filtration are finite elementary abelian 
$p$-groups, and hence $\Out_1(B)$ is a finite $p$-group in that case. 
If $k$ is finite, then $k\tenO B$ is a finite set; in particular, its 
automorphism group as a $k$-algebra is finite. Since $\Out(\bar B)$ is 
isomorphic to a subgroup of finite index in $\Pic(\bar B)$, it follows 
that $\Pic(\bar B)$ is finite, and hence so is $\Pic(B)$, proving (ii). 

For statement (iii), assume that $k$ is an algebraic closure of $\Fp$. 
By (i), every element in $\Out_1(B)$ has finite order, and hence it
suffices to show that every element in $\Pic(\bar B)$ has finite 
order. The result follows from Lemma
\ref{PicFpbar}.
\end{proof}

\section{Examples}

\begin{exa} \label{nilpotentEx}
Let $G$ be a finite group and let $B$ be a nilpotent block of $\OG$ with
a nontrivial defect group $P$. By results of Roggenkamp and Scott 
\cite{RoSc87} (see also Weiss \cite{Weiss}, \cite{Weiss2}), we have a 
canonical group isomorphism
$$\Pic(\OP) = \CL(\OP) \cong \Hom(P,\CO^\times)\rtimes \Out(P)$$
This isomorphism is induced by the map $\Phi$ in Theorem
\ref{thm:endoMorita} applied to $\OP$. (The source algebra $A$ in 
\ref{thm:endoMorita} is in that case $\OP$, and $\Out_P(A)$ is trivial).

By the structure theory of nilpotent blocks from Puig \cite{Punil}
there is a Morita equivalence between $\OP$ and $B$ given by a 
$(B, \OP)$-bimodule $M$ with vertex $\Delta P$ and endopermutation
source $V$. Then the $\CO$-dual $M^*$ has vertex $\Delta P$ and 
endopermutation source $V^*$.  The map sending an $(\OP, \OP)$-bimodule 
$N$ to $M\tenOP N\tenOP M^*$ induces an isomorphism
$$\Pic(\OP)\cong \Pic(B)$$
Thus the map $\Phi$ in Theorem \ref{thm:endoMorita} applied to the 
nilpotent block $B$ sends $\Pic(B)$ isomorphically to the subgroup 
$$[V]\cdot (\Hom(P,\CO^\times)\rtimes\Out(P)) \cdot[V^*]$$
of $D_\CO(P)\rtimes\Out(P)$. The source algebra $A$ is in this case
isomorphic to $\End_\CO(V)\tenO \OP$, and $\Out_P(A)$ is
again trivial (because the inertial quotient of $B$ is trivial). 
It follows that the elements of $\Pic(B)$ correspond under $\Phi$
to elements in $D_\CO(P)\rtimes\Out(P)$ of the form
$$[V] \cdot[\zeta]\cdot [\varphi]\cdot [V^*] =
[\zeta]\cdot [V\ten {^\varphi(V^*)}]\cdot [\varphi]$$
where $[V]$, $[\zeta]$ are the classes in $D_\CO(P)$ of $V$ 
and of the rank $1$ module determined by $\zeta\in$ 
$\Hom(P,\CO^\times)$, respectively, and where $[\varphi]$ is
the class of an automorphism $\varphi$ of $P$ in $\Out(P)$.
If $V$ is not stable under $\varphi$, then the commutator
$$[V\ten {^\varphi(V^*)}] = [ [V],[\varphi]]$$
in $D_\CO(P)\rtimes\Out(P)$ is not trivial. In that case, the
corresponding element of $\Pic(B)$ is given by a bimodule with 
$\Delta\varphi$ as a vertex and a nontrivial, and possibly non-linear, 
endopermutation source. This scenario does arise; see the next example.
\end{exa}

\begin{exa} \label{productExa}
Any Morita equivalence between two $\CO$-algebras $A$, $B$ given
by an $(A,B)$-bimodule $M$ and a $(B,A)$-bimodule $N$
can be interpreted as a self Morita equivalence of the algebra
$A\tenO B$ given by the bimodules $(M\tenO N)_\tau$ and
${_{\tau}(N\tenO M)}$, where $\tau : A\tenO B\to$ $B\tenO A$
is the isomorphism satisfying $\tau(a\ten b)=$ $b\ten a$. Applied to 
block algebras of finite groups, one can use this to construct
self Morita equivalences with endopermutation sources which
need not be linear.

Let $G$, $H$ be finite groups, $b$, $c$ blocks of $\OG$, $\OH$,
respectively, with a common defect group $P$. Set $B=$ $\OGb$ and
$C=$ $\OHc$.  Then $B\tenO C$ is a block of $G\times H$ 
with defect group $P\times P$. 

Let $M$ be a $(B,C)$-bimodule with vertex $\Delta P$ and 
endopermutation source $V$ inducing a Morita equivalence between $B$ and
$C$. Then $M\tenO C$, regarded as a $(B\tenO C,C\tenO C)$-bimodule, 
induces a Morita equivalence between $B\tenO C$ and $C\tenO C$ with 
vertex $\Delta(P\times P)$ and source $V\tenO \CO$. Let $\tau$ be the 
automorphism of $C\tenO C$ defined by $\tau(c\ten c')=$ $c'\ten c$. 
Then $\tau$ restricts to an automorphism of $P\times P$ exchanging the 
two copies of $P$, inducing an automorphism of $\Delta(P\times P)$ in 
the obvious way (and all those restrictions of $\tau$ are again denoted 
by $\tau$).
  
Thus the $(C\tenO C, C\tenO C)$-bimodule $T=$ $(C\tenO C)_\tau$ 
induces a self Morita equivalence with vertex 
$$\Delta\tau=\{((u,v),(v,u))\ |\ u,v\in P\}$$ 
and trivial source.  The $(B\tenO C, B\tenO C)$-bimodule 
$$(M\tenO C)\ten_{C\tenO C} T \ten_{C\tenO C} (M^*\tenO C)$$
induces a self Morita equivalence of $B\tenO C$ with
vertex $\Delta\tau$ and the  $\Delta\tau$-module $V\tenO V^*$ as 
a source (and this is nontrivial if $V$ is nontrivial, and nonlinear 
if $V$ is nonlinear).
\end{exa}


\end{document}